\documentclass[12pt,reqno]{amsart} 
\usepackage[T1]{fontenc}
\usepackage{amsfonts}
\usepackage{amssymb}
\usepackage{mathrsfs}
\usepackage[latin1]{inputenc}
\usepackage{amsmath}
\usepackage{paralist}
\usepackage[english]{babel}
\usepackage{setspace}
\usepackage{bbm}

\newtheorem{theorem}{Theorem}[section]
\newtheorem{lemma}[theorem]{Lemma}
\newtheorem{definition}[theorem]{Definition}
\newtheorem{remark}[theorem]{Remark}
\newtheorem{prop}[theorem]{Proposition}
\newtheorem{example}[theorem]{Example}
\newtheorem{corollary}[theorem]{Corollary}

\numberwithin{equation}{section}

\newcommand{\D}{{\mathbb D}}
\newcommand{\Z}{{\mathbb Z}}
\newcommand{\C}{{\mathbb C}}
\newcommand{\N}{{\mathbb N}}

\newcommand{\cL}{{\mathcal L}}

\newcommand{\cE}{{\mathcal E}}

\newcommand{\cP}{{\mathcal P}}

\newcommand{\cB}{{\mathcal B}}

\newcommand{\cM}{{\mathcal M}}

\newcommand{\be}{\beta}

\newcommand{\si}{\sigma}

\newcommand\Ker{\mathop{\rm Ker}}

\newcommand{à}{\`a}

\begin{document}

\title{Optimal domain of Volterra operators in Korenblum spaces}

\author{Angela\,A. Albanese, Jos\'e Bonet and Werner\,J. Ricker}


\address{ Angela A. Albanese\\
Dipartimento di Matematica e Fisica
``E. De Giorgi''\\
Universit\`a del Salento- C.P.193\\
I-73100 Lecce, Italy}
\email{angela.albanese@unisalento.it}

\address{Jos\'e Bonet \\
Instituto Universitario de Matem\'{a}tica Pura y Aplicada
IUMPA \\
Universitat Polit\`ecnica de Val\`encia \\
E-46071 Valencia, Spain} \email{jbonet@mat.upv.es}

\address{Werner J.  Ricker \\
Math.-Geogr. Fakultät \\
 Katholische Universität
Eichst\"att-Ingol\-stadt \\
D-85072 Eichst\"att, Germany}
\email{werner.ricker@ku.de}

\begin{abstract}
The aim of this article is to study the largest domain space $[T,X]$, whenever it exists, of a given continuous linear operator $T\colon X\to X$, where $X\subseteq H(\D)$ is a Banach space of analytic functions on the open  unit disc $\D\subseteq \C$. That is, $[T,X]\subseteq H(\D)$ is the \textit{largest} Banach space of analytic functions containing $X$ to which $T$ has a continuous, linear, $X$-valued extension $T\colon [T,X]\to X$. The class of operators considered consists of generalized Volterra operators $T$ acting in the Korenblum growth Banach spaces $X:=A^{-\gamma}$, for $\gamma>0$. Previous studies dealt with the classical Cesàro operator $T:=C$ acting in the Hardy spaces $H^p$, $1\leq p<\infty$, \cite{CR}, \cite{CR1}, in $A^{-\gamma}$, \cite{ABR-R}, and more recently, generalized Volterra operators $T$ acting in $X:=H^p$, \cite{BDNS}.
\end{abstract}

\maketitle

\markboth{A.\,A. Albanese, J. Bonet and W.\,J. Ricker}%
{\MakeUppercase{Optimal domain of Volterra operators}}

\section{Introduction and Preliminaries}

Let $X$, $Y$ be Banach spaces (of functions) and $T\colon X\to Y$ be a linear operator, perhaps given by an explicit formula or arising from an inequality which makes $T$ continuous. Situations occur, for instance, whenever there exists $f\not\in X$ such that $Tf\in Y$, when it is desirable to determine the possibility of extending  $T$ beyond its initial domain space $X$ (but still taking its values in $Y$) to a larger space of functions $Z$ containing $X$. Ideally, $Z$ should be the largest or \textit{optimal} domain for the extension of $T$. To make a point, no one would be content with specifying the $c_0(\Z)$-valued Fourier transform operator on, say, $L^2([-\pi,\pi])$; its natural optimal domain (within the realm of function spaces) should be $L^1([-\pi,\pi])$. Of course, more substantial examples exist. As a sample, we point out that
 various classical inequalities (for example, the Sobolev inequality, \cite{A}, \cite{CRN}, \cite{B}, the Hausdorff-Young inequality, \cite{C}, and so on) have been shown to remain valid for larger domain spaces of functions than those in which they were initially formulated. The same is true for well known classes of operators acting between spaces of measurable functions, such as kernel operators, \cite{D}, convolution operators, \cite{E}, Fourier multiplier operators, \cite{F}, and others; see also \cite{G} and the references there in. 
 
 The theme of this paper is analogous to that described above. However, the setting is now rather different and will require new methods and techniques. The aim is 
 to investigate the \textit{extension procedure} alluded to above  for linear operators acting between Banach spaces of \textit{analytic functions} on the open unit disc $\D\subseteq \C$. The ambient space involved is now the Fr\'echet space $H(\D)$ consisting of all analytic functions on $\D$, equipped with the topology of   uniform convergence on the  compact subsets of $\D$. Given is a Banach space $X\subseteq H(\D)$, containing the polynomials, for which the natural inclusion map is continuous and a continuous linear operator $T\colon X\to X$. The fundamental question is as follows: do there exist further Banach spaces $Z\subseteq H(\D)$ such that $X\subseteq Z$ and $T\colon X\to X$ has a continuous, linear, $X$-valued extension $T\colon Z\to X$? If so, is there a \textit{largest} such space $Z$, the so called \textit{optimal domain space} of $T$, and can it be identified? In some cases the optimal domain $Z$ is actually a \textit{new} Banach space of analytic functions not available before, which generates an interest to determine structural properties of $Z$. Moreover, the factorization of $T\colon X\to X$ as $T_Z\circ j$, where $j\colon X\to Z $ is the inclusion map and $T_Z\colon Z\to X$ is the extension of $T$, provides a technique to examine operator ideal properties such as compactness, weak compactness, etc.
 The above question was first introduced (and solved) for the classical Cesàro operator $C:=T$ acting in the family of Hardy spaces $X:=H^p$, for $1\leq p<\infty$, \cite{CR}, \cite{CR1}, and in $X:=A^{-\gamma}$, for $\gamma>0$, \cite{ABR-R}. In the recent article \cite{BDNS} a similar investigation was undertaken for the family of generalized Volterra operators $T:=V_g$, with $g\in BMO\cap H^2$, acting in $H^p$ for $1\leq p<\infty$. If $g(z)=-{\rm Log}(1-z)$, for $z\in\D$, then $V_g(f)(z)=zC(f)(z)$, for $f\in H^p$ and $z\in\D$.

In this article we undertake a detailed study  of the optimal domain space for the particular class of (Volterra) integral operators given by  $(V_gf)(z)=\int_0^z f(\xi)g'(\xi)\,d\xi$, for $z\in \D$ and $f\in X$, with the function $g\in H(\D)$ coming from  the Bloch spaces $\cB$ and $\cB_0$, and where $X$ varies through the family of Korenblum growth Banach spaces $A^{-\gamma}$ and $A^{-\gamma}_0$, for $\gamma>0$; see Section 3 for the definitions and notation. We caution the reader that the above operators $V_g$ correspond to those which are denoted by $T_g$ in \cite{BDNS} and vice versa. The notation $T_g$ is used in this paper for another operator defined in Section 5. 

It is clear from the above discussion that there is a need to clearly formulate various concepts and to identify a precise abstract setting for the ``optimal extension procedure''; see Definition \ref{Def-B}. This is carried out in Section 2, which not only creates the framework for this paper but, also for possible future research on this topic in general. In particular, the basic properties of the optimal domain space $[T,X]$ for a continuous linear operator $T\colon X\to X$, with $X\subseteq H(\D)$ a Banach space of analytic functions on $\D$, are established in Propositions \ref{P1}--\ref{P3}.

In Section 3 we concentrate on the \textit{Korenblum growth spaces} $A^{-\gamma}$ and $A^{-\gamma}_0$, for $\gamma>0$, all of which are Banach spaces of analytic functions on $\D$. The class of operators which  we consider consists of the Volterra operators $V_g$, for $g\in\cB$ (resp. $g\in \cB_0$), defined above, which are necessarily continuous (cf. Propositions \ref{P.Con} and \ref{P.Comp}) and whose optimal domain space $[V_g,A^{-\gamma}]$ (resp. $[V_g,A^{-\gamma}_0]$) turns out to also be a Banach space of analytic functions on $\D$;  see Proposition \ref{P.D1} and Corollary \ref{CoroE}. An alternative description of these optimal domain spaces is given in Proposition \ref{P.Description}; see also Remark \ref{R1}. In Example \ref{E1} a class of functions $g\in\cB$ is presented for which $A^{-\gamma}\subsetneqq [V_g,A^{-\gamma}]$, that is, $V_g\colon A^{-\gamma}\to A^{-\gamma}$ has a \textit{genuine} $A^{-\gamma}$-valued extension to $[V_g,A^{-\gamma}]$.

Each function $h\in H(\D)$ generates the continuous linear operator $M_h\colon H(\D)\to H(\D)$ of pointwise multiplication by $h$. Given a pair of Banach spaces of analytic functions on $\D$, say $X$ and $Y$, a function $h\in H(\D)$ is called a \textit{multiplier} for $X$, $Y$ if $M_h(X)\subseteq Y$. The space of all such  multipliers $h$ is denoted by $\cM(X,Y)$ or, if $X=Y$, simply by $\cM(X)$. Since $[T,X]$ is typically a Banach space of analytic functions on $\D$, it is natural to study its multiplier space $\cM([T,X])$. For the optimal domain space $[C,H^p]$ of the Cesàro operator $C$, for $1\leq p<\infty$, it is known that $\cM([C,H^p])=H^p$, \cite[Theorem 3.7]{CR}, and for the generalized Volterra operator $T_g$, whenever $g\in BMO\cap H^2$, considered in \cite{BDNS} that $\cM([T_g, H^p])=H^\infty$ for each $1\leq p<\infty$; see Theorem 4 there. In Section 4 we establish that the multiplier space $\cM([V_g, A^{-\gamma}])=H^\infty$ and $\cM([V_g, A^{-\gamma}_0])=H^\infty$ for each $\gamma>0$ and $g\in\cB$; see Proposition \ref{P.Mult} and Corollary \ref{coroG}. For the Korenblum spaces themselves it is also shown, whenever $0<\gamma<\delta$, that $\cM(A^{-\gamma}, A^{-\delta})=A^{-(\delta-\gamma)}=\cM(A^{-\gamma}_0, A^{-\delta}_0)$; see Proposition \ref{PH=} and also \cite[Proposition 5]{BDNS}, \cite[Proposition 3.1]{Cont-Diaz}. Moreover, Proposition \ref{L_M} reveals that $\cM(A^{-\gamma}, A^{-\delta}_0)=A^{-(\delta-\gamma)}_0$.

In the final Section 5 we turn  our  attention to the \textit{weighted} Banach spaces of analytic functions $H^\infty_v$ and $H^0_v$ on $\D$ for certain weight functions $v\colon [0,1)\to (0,\infty)$. The family of operators considered consists of the Cesàro operator $C$ together with $V_g$ for certain functions $g$. Aleman and Persson have made a detailed study of generalized Cesàro operators acting in various Banach spaces of analytic functions, \cite{AP}, \cite{AleSi1}, \cite{P}, including the spaces $A^{-\gamma}$ and $A^{-\gamma}_0$; see also \cite{ABR-R}. We begin by considering properties of the optimal domain spaces $[V_g, H^\infty_v]$ and $[T_g, H^\infty_v]$ and also of  $[V_g, H^0_v]$ and $[T_g, H^0_v]$, where $T_g$ is given by \eqref{eq.T}; see Proposition \ref{P.Ugu}. Of particular interest is the Cesàro operator $C=T_{g_0}$ with $g_0(z)=-{\rm Log}(1-z)$, for $z\in\D$. It turns out that the optimal domain spaces $[C,A^{-\gamma}]$ and  $[V_{g_0},A^{-\gamma}]$ with $\gamma>0$ (resp. $[C,A^{-\gamma}_0]$ and $[V_{g_0},A^{-\gamma}_0]$) are \textit{genuinely larger} than their initial domain spaces $A^{-\gamma}$ (resp. $A^{-\gamma}_0$); see Proposition \ref{P.DomainC}. Also Propositions \ref{I}--\ref{P.Notclosed} and Proposition \ref{P.B0} exhibit related features. In Proposition \ref{P.NoCl} a large class of functions $g$ in $\cB$ is exhibited for which both of the operators $V_g\colon A^{-\gamma}\to A^{-\gamma}$ and $V_g\colon A^{-\gamma}_0\to A^{-\gamma}_0 $ fail to have closed range; see Example \ref{E2} for some particular functions $g$ to which Proposition \ref{P.NoCl} applies.

\medskip

Given Fr\'echet spaces $X$ and $Y$,  denote by $\cL(X,Y)$ the space of all continuous linear operators from $X$ into $Y$. For the case when $X=Y$,  we simply write $\cL(X)$ for $\cL(X,X)$. If both $X$ and $Y$ are Banach spaces then, for the operator norm $\|T\|_{X\to Y}:=\sup_{\|x\|_X\leq 1}\|Tx\|_Y$, the space $\cL(X,Y)$ is a Banach space.
Equipped with the topology of pointwise convergence on a Fr\'echet space $X$ (i.e., the strong operator topology $\tau_s$) the quasi-complete locally convex Hausdorff space $\cL(X)$ is denoted by $\cL_s(X)$. The range $T(X):=\{Tx:\ x\in X\}$ of $T\in \cL(X)$ is also denoted by ${\rm Im}(T)$. Furthermore, $\Ker(T):=\{x\in X:\, Tx=0\}$.

Let $X$ be a Fr\'echet space.
The identity operator on $X$ is written as $I$.  The \textit{transpose operator} of $T\in \cL(X)$ is denoted by  $T^*$; it acts from the topological dual space $X^*:=\cL(X,\C)$ of $X$ into itself. Denote by $X^*_\si$ (resp., by $X^*_\beta$) the space $X^*$ equipped with the weak* topology $\si(X^*,X)$ (resp., with the strong topology $\beta(X^*,X)$). It is known that $X^*_\si$ is quasicomplete with $T^*\in \cL(X^*_\si)$ and that $T^*\in \cL(X^*_\be)$,  \cite[p.134]{24}. The bi-transpose operator $(T^*)^*$ of $T$ is simply denoted by $T^{**}$ and belongs to $\cL((X^*_\beta)^*_\beta)$. In the event that $X$ is a Banach space, both $X^*_\beta$ (denoted simply by $X^*$) and $(X^*_\beta)^*_\beta$ (denoted simply by $X^{**}$) are again Banach spaces. The dual norm in $X^*$ is given by $\|x^*\|:=\sup_{\|x\|\leq 1}|\langle x,x^*\rangle|$, for $x^*\in X^*$.

A linear map $T\colon X\to Y$, with $X$ and $Y$ Banach spaces, is called \textit{compact} if  $T(B_X)$ is a relatively compact set in $Y$, where $B_X$ denotes the closed unit ball  of $X$. It is routine to show that necessarily $T\in \cL(X,Y)$.

\section{General properties concerning optimal domains}
Denote by $H(\D)$ the space of all analytic functions on the open unit disc $\D:=\{z\in\C\,:\, |z|<1\}$.
The space $H(\D)$ is equipped with the topology $\tau_c$ of uniform convergence on the compact subsets of  $\D$. According to \cite[\S 27.3(3)]{23} the space $H(\D)$ is a Fr\'echet-Montel space. An increasing  family of norms generating $\tau_c$ is given, for each $0<r<1$, by
\begin{equation}\label{eq.norme-sup}
	q_r(f):=\sup_{|z|\leq r}|f(z)|,\quad f\in H(\D).
\end{equation}
We identify  a function $f\in H(\D)$ with its sequence of Taylor coefficients $\hat{f}:=(\hat{f}(n))_{n\in\N_0}$ (i.e., $\hat{f}(n):=\frac{f^{(n)}(0)}{n!}$,  for $n\in\N_0$),  so that $f(z)=\sum_{n=0}^\infty \hat{f}(n)z^n$, for $z\in\D$. Given $z\in\D$, it is clear from \eqref{eq.norme-sup} that each evaluation functional $\delta_z\colon  f\mapsto f(z)$, for $f\in H(\D)$, is linear and continuous, that is, $\delta_z\in H(\D)^*$.

A vector space $X\subseteq H(\D)$ is called a \textit{Banach space of analytic functions} on $\D$ if it is a Banach space relative to a norm $\|\cdot \|_X$ and the natural inclusion
of $X$ into $H(\D)$ is continuous.

\begin{lemma}\label{Lemma A} Let $X$ be a Banach space such that $X\subseteq H(\D)$ as linear spaces. Then the natural inclusion map of $X$ into $H(\D)$ is continuous if and only if $\{\delta_z:\ z\in \D\}\subseteq X^*$.
	\end{lemma}

\begin{proof} Suppose that $X\subseteq H(\D)$ continuously. Fix $z\in \D$. Let $(f_n)_{n\in\N}\subseteq X$ satisfy $\|f_n\|_X\to 0$ as $n\to\infty$. Then also $f_n\to 0$ in $H(\D)$ as $n\to\infty$. Since $\delta_z\in H(\D)^*$ it follows that $f_n(z)=\langle f_n,\delta_z\rangle\to 0$ in $\C$  as $n\to\infty$. This shows that $\delta_z\colon X\to \C$ is continuous, that is $\delta_z\in X^*$.
	
	Suppose now that $\{\delta_z:\, z\in \D\}\subseteq X^*$. Let $(f_n)_{n\in\N}\subseteq X$ satisfy $f_n\to 0$ in $X$ and $f_n\to g$ in $H(\D)$, for some $g\in H(\D)$,  as $n\to\infty$. Fix $z\in \D$. Since $\delta_z\in X^*$, it follows that $f_n(z)=\langle f_n,\delta_z\rangle\to 0$ in $\C$ as $n\to\infty$. Also $\delta_z\in H(\D)^*$ and so $\langle f_n,\delta_z\rangle\to \langle g, \delta_z\rangle$, that is, $f_n(z)\to g(z)$ in $\C$ as $n\to\infty$. Hence, $g(z)=0$. Since $z\in\D$ is arbitrary, we may conclude that $g=0$ in $H(\D)$. Accordingly, the inclusion map $X\subseteq H(\D)$ is a closed, linear map which implies its continuity by the closed graph theorem for Fr\'echet spaces.
	\end{proof}

Let $(X,\|\cdot\|_X)$ be a Banach space of analytic functions on $\D$ and $T\colon H(\D)\to H(\D)$ be a continuous,  linear operator satisfying $T(X)\subseteq X$. By the closed graph theorem it follows that $T\in \cL(X)$. Indeed, if $(f_n)_{n\in\N}\subseteq X$ satisfies both $f_n\to 0 $ in $X$ and $Tf_n\to g$ in $X$ for some $g\in X$ as $n\to\infty$, then also $f_n\to 0$ in $H(\D)$ and $Tf_n\to g$ in $H(\D)$ as $n\to\infty$. But, $T\in \cL(H(\D))$. This implies that $Tf_n\to 0$ in $H(\D)$ as $n\to\infty$. So, $g=0$. Hence, $T\colon X\to X$ is a closed operator.

The following notion first occurs in \cite{CR} for the Cesàro operator $T:=C$ acting in $X=H^p$.

\begin{definition}\label{Def-B} The \textit{optimal domain} of an operator $T\in\cL(H(\D))$ satisfying $T(X)\subseteq X$, with $X$ a Banach space of analytic functions on $\D$, is the linear subspace
 \[
 [T,X]:=\{f\in H(\D)\colon\ Tf\in X\}
 \]
 of $H(\D)$  endowed with the semi-norm
 \[
 \|f\|_{[T,X]}:=\|Tf\|_X,\quad f\in [T,X].
 \]
 \end{definition}

Let $X$ be a Banach space of analytic functions on $\D$ and $T\in \cL(H(\D))$ satisfy $T(X)\subseteq X$. A Banach space $Z$ of analytic functions on $\D$ is said to be a $(T,X)$-\textit{admissible space} if it satisfies
\begin{equation}\label{Sigma}
	T(Z)\subseteq X\subseteq Z.
\end{equation}
Then the map $T_Z\colon Z\to X$ defined by $T_Zf:=Tf$, for $f\in Z$, is an $X$-valued, linear  extension of $T\colon X\to X$ which, in this notation, can  also be written as $T_X$.

\begin{lemma}\label{LemmaC} Let $X$ be a Banach space of analytic functions on $\D$ and $T\in \cL(H(\D))$ satisfy $T(X)\subseteq X$. Let $Z$ be a $(T,X)$-admissible space.
\begin{itemize}
	\item[\rm (i)] The inclusion $X\subseteq Z$ is continuous.
	\item[\rm (ii)] The linear  map $T_Z\colon Z\to X$ is continuous and its restriction to $X$ coincides with $T\colon X\to X$.
\end{itemize}
	\end{lemma}

\begin{proof} (i) Let $\Phi\colon X\to Z$ denote the natural inclusion map. Suppose that $(f_n)_{n\in\N}\subseteq X$ satisfies $f_n\to 0$ in $X$ and $\Phi f_n\to g$ in $Z$, for some $g\in Z$. Lemma \ref{Lemma A} implies that $\{\delta_z:\, z\in \D\}\subseteq X^*$. Accordingly, for any $w\in\D$, we have $f_n(w)=\langle f_n,\delta_w\rangle\to 0$ in $\C$ as $n\to\infty$.  Lemma \ref{Lemma A} also implies that $\{\delta_z:\, z\in \D\}\subseteq Z^*$ and so $f_n(w)=\langle \Phi f_n, \delta_w\rangle\to \langle g,\delta_w\rangle=g(w)$ in $\C$ as $n\to\infty$. Accordingly, $g(w)=0$. Since $w\in\D$ is arbitrary, we have $g=0$. The closed graph theorem ensures that $\Phi$ is continuous.
	
	(ii) It is clear from the definition of $T_Z$ that it is linear and its restriction to $X$ coincides with $T\colon X\to X$. Let $(p_n)_{n\in\N}\subseteq Z$ satisfy $p_n\to 0$ in $Z$ and $T_Zp_n\to q$ in $X$, for some $q\in X$. Since $(T_Zp_n)_{n\in\N}=(Tp_n)_{n\in\N}\subseteq X$, it follows from part (i) that $Tp_n\to q$ in $Z$ as $n\to\infty$. Moreover, $T(Z)\subseteq Z$ by \eqref{Sigma} and so, via the discussion prior to Definition  \ref{Def-B}, we know that $T\in \cL(Z)$. Hence, $p_n\to 0$ in $Z$ implies that $Tp_n\to 0$ in $Z$ and so we can conclude that $q=0$. By the closed graph theorem $T_Z\colon Z\to X$ is continuous.
\end{proof}

Lemma \ref{LemmaC} states, whenever $Z$ is a $(T,X)$-admissible space,  that $T_Z\colon Z\to X$ is a \textit{continuous}, $X$-valued, linear extension of $T\colon X\to X$ from $X$ to $Z$. Our aim is to show that the choice $Z=[T,X]$ yields the \textit{largest} $(T,X)$-admissible space, which thereby justifies the terminology ``optimal domain of $T$'' used in  Definition \ref{Def-B}. First we require some preparation.

 \begin{prop}\label{P1} Let $X$ be a Banach space of analytic functions on $\D$ and $T\in\cL(H(\D))$  satisfy $T(X)\subseteq X$, in which case $T\in \cL(X)$. Then the following properties are satisfied.
 	\begin{itemize}
 		\item[\rm (i)] $[T,X]\subseteq H(\D)$.
 		\item[\rm (ii)] Suppose that $T\in \cL(H(\D))$ is injective. Then  $([T,X], \|\cdot\|_{[T,X]})$ is a normed space.
 		\item[\rm (iii)] $X\subseteq [T,X]$ with a continuous inclusion from $X$ into the semi-normed space $[T,X]$.
 		\item[\rm (iv)] The operator $T_{[T,X]}\colon [T,X]\to X$,defined by  $f\mapsto Tf$, is linear and continuous from the semi-normed space $[T,X]$ into $X$.
 		\item[\rm (v)] Let $Y$ be a closed subspace of $X$ satisfying $T(Y)\subseteq Y$. Then $Y$ is  a Banach space of analytic functions on $\D$ and $[T,Y]$ is a closed subspace of $[T,X]$.
 		\item[\rm (vi)] $\Ker(T_X)=X\cap\Ker(T_{[T,X]})\subseteq \Ker(T_{[T,X]})\subseteq \Ker(T)$ with $T\in \cL(H(\D))$.
 	\end{itemize}
 	\end{prop}

 \begin{proof} 	(i) Clear from the  definition.
 	
 	(ii) Let $f\in [T,X]$ satisfy $\|f\|_{[T,X]}=0$. Then $f\in H(\D)$ and  $\|Tf\|_X=\|f\|_{[T,X]}=0$. Since $T$ is injective  when it acts on $H(\D)$, it follows that $f=0$. So, $\|\cdot \|_{[T,X]}$ is a norm on $[T,X]$.
 	
 	(iii) Since $T(X)\subseteq X$, it is clear that $X\subseteq [T,X]$. On the other hand, for every  $f\in X$ we have that  $\|f\|_{[T,X]}=\|Tf\|_X\leq \|T\|_{X\to X}\|f\|_X$. So, the inclusion $X\subseteq [T,X]$ is continuous.

 	(iv) Clearly, $T_{[T,X]}$ is linear. Moreover, given  $f\in [T,X]$ we have that $\|T_{[T,X]}f\|_X=\|Tf\|_X=\|f\|_{[T,X]}$. So, the  operator $T_{[T,X]}\colon [T,X]\to X$ is surely continuous.
 	
 	(v)  Clearly $Y$, endowed with the norm $\|\cdot\|_X$ inherited from $X$, is a Banach space of analytic functions on $\D$. Since $T(Y)\subseteq Y$, it is clear that the restriction  $T\colon Y\to Y$ of $T\in \cL(X)$ is continuous. Let $f\in [T,Y]$, that is, $f\in H(\D)$ satisfies $Tf\in Y$. Then also $Tf\in X$ and hence, $f\in [T,X]$. So, $[T,Y]\subseteq [T,X]$.
 	
 	Now, let $(f_n)_{n\in\N}\subseteq [T,Y]$ be a sequence convergent to some $f$ in $[T,X]$. This means that
 	\[
 	\|f_n-f\|_{[T,X]}=\|Tf_n-Tf\|_X\to 0,\quad \mbox{ as } n\to\infty.
 	\]
 	But, $(Tf_n)_{n\in\N}\subseteq Y$ and $Y$ is a closed subspace of $X$. It follows that $Tf\in Y$ and hence, that $f\in [T,Y]$. This shows that $[T,Y]$ is a closed subspace of $[T,X]$.
 	
 	(vi) Clear from the  definitions involved.
 	\end{proof}

 Combining Lemma \ref{LemmaC} and Proposition \ref{P1} yields the following result.

 \begin{prop}\label{PropD} Let $X$ be a Banach space of analytic functions on $\D$ and $T\in \cL(H(\D))$ satisfy $T(X)\subseteq X$. Suppose that $[T,X]$ is a Banach space. Then, amongst all $(T,X)$-admissible spaces, the optimal domain space $[T,X]$ of $T$ is the largest one.
 \end{prop}

\begin{proof} By hypothesis $[T,X]\subseteq H(\D)$ is a Banach space of analytic functions on $\D$. Moreover, $X\subseteq [T,X]$ continuously; see Proposition \ref{P1}(iii). Let $f\in [T,X]$. According to Definition \ref{Def-B} the function $f\in H(\D)$ and $Tf\in X$. Hence, $T$ maps $[T,X]$ into $X$. So, \eqref{Sigma} is satisfied with $[T,X]$ in place of $Z$, that is, $[T,X]$ is a $(T,X)$-admissible space.
	
	Let $Z$ be any $(T,X)$-admissible space. Given $f\in Z$ we have $f\in H(\D)$ and, via \eqref{Sigma}, necessarily $Tf\in X$. Hence, $f\in [T,X]$; see Definition \ref{Def-B}. This shows that $Z\subseteq [T,X]$.
\end{proof}

Proposition \ref{PropD} shows the importance of being able to decide when $[T,X]$ is a Banach space. The following result provides a sufficient condition.

 \begin{prop}\label{P2}Let $X$ be a Banach space of analytic functions on $\D$ and $T\in\cL(H(\D))$ be an isomorphism of $H(\D)$ satisfying $T(X)\subseteq X$. Then $([T,X],\|\cdot\|_{[T,X]})$ is a Banach space.
 	\end{prop}

 \begin{proof} Let $(f_n)_{n\in\N}\subseteq [T,X]$ be a Cauchy sequence in $[T,X]$. Then $(Tf_n)_{n\in\N}$ is a Cauchy sequence in $X$. Since $X$ is a Banach space, there exists $g\in X$ such that $Tf_n\to g$ in $X$ as $n\to \infty$ and hence,  $Tf_n\to g$  in $H(\D)$ as $n\to \infty$. Since $T\colon H(\D)\to H(\D)$ is an isomorphism, it follows that $f_n\to T^{-1}g$ in $H(\D)$ as $n\to\infty$. Accordingly, $h:=T^{-1}g\in H(\D)$ satisfies $Th=g\in X$. So, $h\in [T,X]$. On the other hand, $\|f_n-h\|_{[T,X]}=\|Tf_n-Th\|_X=\|Tf_n-g\|_X\to 0$ as $n\to\infty$, that is, $f_n\to h$ in $[T,X]$ as $n\to\infty$.
 	\end{proof}

 In the following result we analyze when the inclusion  $[T,X]\subseteq H(\D)$ (see Proposition \ref{P1}(i)) is continuous, equivalently, when $[T,X]$ is a Banach space.

 \begin{prop}\label{P3}  Let $X$ be a Banach space of analytic functions on $\D$ and $T\in\cL(H(\D))$ be an injective operator satisfying $T(X)\subseteq X$, in which case $\|\cdot\|_{[T,X]}$ is a norm. The following three properties are equivalent.
 	\begin{itemize}
 		\item[\rm (i)] The inclusion $[T,X]\subseteq H(\D)$ is continuous.
 		\item[\rm (ii)] The evaluation functionals $\{\delta_z:\, z\in \D\}\subseteq [T,X]^*$ and,  for each compact subset $K$ of $\D$, the set $\{\delta_w\colon\ w\in K\}$ is bounded in $[T,X]^*$.
 		\item[\rm (iii)] $([T,X],\|\cdot\|_{[T,X]})$ is a Banach space.
 	\end{itemize}
 If any one of {\rm (i)--(iii)} is satisfied, then $[T,X]$ is a Banach space of analytic functions on $\D$ and $\{\delta_z:\, z\in \D\}\subseteq [T,X]^*$.
 	\end{prop}

 \begin{proof} (i)$\Rightarrow$(ii) Let $z\in \D$. The evaluation functional $\delta_z\colon H(\D)\to\C$ is known to be continuous. By (i) it follows that  $\delta_z\colon [T,X]\to\C$ is also continuous, that is, $\delta_z\in [T,X]^*$.
 	
 	For a fixed compact subset $K$ of $\D$, by the assumption of (i) and the definition of the topology $\tau_c$ in $H(\D)$ there exists $C>0$ such that
 	\[
 	\sup_{z\in K}|f(z)|\leq C\|f\|_{[T,X]},\quad f\in [T,X].
 	\]
 	Fix $w\in K$. Given $f\in [T,X]$ satisfying $\|f\|_{[T,X]}\leq 1$, it follows from the previous inequality  that
 	\[
 	|\delta_w(f)|=|f(w)|\leq C.
 	\]
 	This implies that $\delta_w\in [T,X]^*$ and   $\|\delta_w\|_{[T,X]^*}\leq C$. So,   $\{\delta_w\colon\ w\in K\}$ is a bounded subset of $[T,X]^*$.
 	
 	(ii)$\Rightarrow$(i) Let $K$ be a compact subset  of $\D$. Since    $\{\delta_z\colon\ z\in K\}$ is a bounded subset of $[T,X]^*$, there exists $C>0$ such that $\|\delta_w\|_{[T,X]^*}\leq C$ for every $w\in K$. It follows, for every $w\in K$, that
 	\[
 	\left|\frac{f(w)}{\|f\|_{[T,X]}}\right|=\left|\delta_w\left(\frac{f}{\|f\|_{[T,X]}}\right)\right|\leq C,  \quad f\in [T,X]\setminus\{0\},
 	\]
 	and hence, that
 	\[
 	|f(w)|\leq C \|f\|_{[T,X]}.
 	\]
 	Accordingly, $\sup_{w\in K}|f(w)|\leq C\|f\|_{[T,X]}$ for every $f\in [T,X]$. Since $K$ is an arbitrary compact subset of $\D$, this implies that the inclusion $[T,X]\subseteq H(\D)$ is continuous.
 	
 	(i)$\Rightarrow$(iii) Let $(f_n)_{n\in\N}\subseteq [T,X]$ be a Cauchy sequence in $[T,X]$, that is, $\|f_n-f_m\|_{[T,x]}=\|Tf_n-Tf_m\|_X\to 0$ as $n,m\to\infty$.  Accordingly, $(Tf_n)_{n\in\N}\subseteq X$ is a Cauchy sequence in $X$ and hence, there exists $g\in X$ such that $Tf_n\to g$ in $X$ as $n\to\infty$. Then also $Tf_n\to g$ in $H(\D)$ as $n\to\infty$. On the other hand, by the hypothesis of (i) the sequence  $(f_n)_{n\in\N}\subseteq H(\D)$ is also a Cauchy sequence in $H(\D)$. Therefore, there exists $f\in H(\D)$ such that $f_n\to f$ in $H(\D)$ as $n\to\infty$. Since $T\in \cL(H(\D))$, it follows that $Tf_n\to Tf$ in $H(\D)$ as $n\to\infty$. So, $Tf=g$ in $H(\D)$. But,   $g\in X$ and so $Tf\in X$,  that is, $f\in [T,X]$. Moreover, $\|f_n-f\|_{[T,X]}=\|Tf_n-Tf\|_X=\|Tf_n-g\|_X\to 0$ as $n\to\infty$, that is, $f_n\to f$ in $[T,X]$. Accordingly, $[T,X]$ is a Banach space relative to $\|\cdot\|_{[T,X]}$.
 	
 		(iii)$\Rightarrow$(i)  We first prove that the natural inclusion $J\colon [T,X]\to H(\D)$, defined by $f\mapsto Jf:=f$, has a closed graph. To see this, let $(f_n)_{n\in\N}\subseteq [T,X]$ satisfy $f_n\to f$ in $[T,X]$ and $Jf_n=f_n\to g$ in $H(\D)$ for some $g\in H(\D)$ as $n\to\infty$. The fact that  $(f_n)_{n\in\N}\subseteq [T,X]$ converges to $f$ in $[T,X]$ implies that $Tf_n\to Tf$ in $X$ as $n\to\infty$, and hence, that $Tf_n\to Tf$  in $H(\D)$ as $n\to\infty$. On the other hand, $T\in \cL(H(\D))$ implies that  $Tf_n\to Tg$ in $H(\D)$ as $n\to\infty$. Therefore, $Tf=Tg$. Since $T$ is injective, it follows that $f=g$. So, $J$ has a closed graph.
 		
 		By property (iii)  the space $[T,X]$ is a Banach space. Moreover, $H(\D)$ is a Fr\'echet space. So, the continuity of the inclusion $J\colon [T,X]\to H(\D)$ follows by the closed graph theorem for Fr\'echet spaces.
 		
 		Finally, if any one of (i)--(iii) is satisfied, then $[T,X]$ is a Banach space with $[T,X]\subseteq H(\D)$ and the inclusion is continuous. Hence, $[T,X]$ is a Banach space of analytic functions on $\D$.
 	\end{proof}

 We now  establish further properties of  the space $([T,X],\|\cdot\|_{[T,X]})$ and of the inclusion $X\subseteq [T,X]$ under suitable assumptions.

 \begin{prop}\label{P4}Let $X$ be a Banach space of analytic functions on $\D$ and $T\in\cL(H(\D))$ be an isomorphism  satisfying $T(X)\subseteq X$. Then $([T,X],\|\cdot\|_{[T,X]})$ is isometrically isomorphic to $X$.
 \end{prop}

\begin{proof} The injectivity of $T\in\cL(H(\D))$ clearly implies that the operator $T\colon [T,X]\to X$, given by  $f\mapsto Tf$, is injective. Moreover, the fact that $T\in\cL(H(\D))$ is surjective  implies that the operator $T\colon [T,X]\to X$, is surjective. Indeed, given $g\in X\subseteq H(\D)$ the function  $T^{-1}g\in H(\D)$ exists and satisfies  $T(T^{-1}g)=g$. Accordingly, $T^{-1}g\in [T,X]$ and $T(T^{-1}g)=g$. On the other hand, $\|Tf\|_X=\|f\|_{[T,X]}$ for every $f\in [T,X]$. So, $T\colon [T,X]\to X$ is a surjective isometry. In particular, $([T,X],\|\cdot\|_{[T,X]})$ is isometrically isomorphic to $X$.
	\end{proof}

\begin{prop}\label{P5} Let $X$ be a Banach space of analytic functions on $\D$ and $T\in\cL(H(\D))$ be an injective operator such that  $T(X)\subseteq X$ and  $([T,X],\|\cdot\|_{[T,X]})$ is a Banach space. The following properties are equivalent.
	\begin{itemize}
		\item[\rm (i)] The continuous linear operator $T\colon X\to X$ has closed range in $X$.
		\item[\rm (ii)] The natural inclusion $J_{[T,X]}\colon X\to [T,X]$, defined by $f\mapsto J_{[T,X]}f:=f$,  has closed range in $[T,X]$.
	\end{itemize}
\end{prop}

\begin{proof} (i)$\Rightarrow$(ii)
	Since $T(X)\subseteq X$ is a closed subspace of $X$, necessarily $(T(X),\|\cdot\|_X)$ is a Banach space. So, an application of  the open mapping theorem implies that the continuous, bijective operator $T\colon X\to T(X)$ is an isomorphism. Accordingly,   there exists $c>0$ such that
	\begin{equation}\label{eq.closed}
		\|Tf\|_X\geq c\|f\|_X, \quad f\in X.
		\end{equation}
	This implies that the inclusion $J_{[T,X]}\colon X\to [T,X]$ has closed range. Indeed, from \eqref{eq.closed} it follows that
	\[
	\|J_{[T,X]}f\|_{[T,X]}=\|f\|_{[T,X]}=\|Tf\|_X\geq c\|f\|_X, \quad f\in X,
	\]
	with $J_{[T,X]}$ continuous (cf. Proposition \ref{P1}(iii)), which yields that the inclusion $J_{[T,X]}\colon X\to [T,X]$ has closed range.
	
	(ii)$\Rightarrow$(i) Since the inclusion $J_{[T,X]}\colon X\to [T,X]$ has closed range, the linear subspace $J_{[T,X]}(X)=X$ is  closed in the Banach space $[T,X]$. So, $(X, \|\cdot\|_{[T,X]})$ is a Banach space. In view of the continuity of $J_{[T,X]}\colon (X,\|\cdot\|_X)\to ([T,X], \|\cdot\|_{[T,X]})$, we oberve that the identity operator $L\colon (X,\|\cdot\|_X)\to (X,\|\cdot\|_{[T,X]})$ is continuous and bijective. Therefore, we can apply the open mapping theorem to conclude that  $L\colon (X,\|\cdot\|_X)\to (X,\|\cdot\|_{[T,X]})$ is an isomorphism. Accordingly, there exists $c>0$ such that $\|Lf\|_{[T,X]}\geq c\|f\|_X$ for every $f\in X$, and hence,
	\[
	\|Tf\|_X=\|f\|_{[T,X]}=\|Lf\|_{[T,X]}\geq c\|f\|_X,\quad f\in X.
	\]
	This implies that the operator $T\colon X\to X$ has closed range.
	\end{proof}

An immediate consequence of Proposition \ref{P5} is the following fact.

\begin{corollary}\label{C1}
Let $X$ be a Banach space of analytic functions on $\D$ and $T\in\cL(H(\D))$ be an injective operator such that  $T(X)\subseteq X$ and $([T,X],\|\cdot\|_{ [T,X]})$ is a Banach space.
\begin{itemize}
		\item[\rm (i)] Suppose that the natural inclusion $J_{[T,X]}\colon X\to [T,X]$ is surjective. Then  the operator $T\colon X\to X$ has closed range in $X$.
	\item[\rm (ii)] Suppose that the operator $T\colon X\to X$ fails to have closed range in $X$. Then $X$ is a proper subspace of $[T,X]$.
\end{itemize}
\end{corollary}

\begin{proof} 
(i)
The inclusion $J_{[T,X]}\colon X\to [T,X]$ is a continuous bijection of $X$ onto its range in $[T,X]$ and hence, it is an isomorphism by the open mapping theorem. Accordingly, $J_{[T,X]}$ has a closed range in $[T,X]$. Then Proposition \ref{P5} implies that $T\colon X\to X$ has a closed range in $X$.

(ii) Suppose that $X=[T,X]$. Then the inclusion $J_{[T,X]}\colon X\to [T,X]$ is surjective. Since $([T,X],\|\cdot\|_{ [T,X]})$ is a Banach space,  the open mapping theorem again implies that  $J_{[T,X]}\colon X\to [T,X]$ is an isomorphism and hence, it has a closed range in $[T,X]$. Therefore, by Proposition  \ref{P5}, also the operator $T\colon X\to X$ has  closed range in $X$; a contradiction. So, $X\subsetneqq [T,X]$. \end{proof}

\section{Optimal domain of Volterra operators on Korenblum growth Banach spaces}

Let $g\in H(\D)$ be a non-constant function. The \textit{Volterra operator} $V_g\colon H(\D)\to H(\D)$ is the linear operator defined by
\begin{equation}\label{eq.V}
	(V_gf)(z):=\int_0^z f(\xi)g'(\xi)\,d\xi,\quad f\in H(\D), \ z\in\D.
	\end{equation}
The operator $V_g$ acts continuously in $H(\D)$. Moreover, $V_g$ is \textit{injective} on $H(\D)$, \cite{AND}, but, not surjective because $(V_gf)(0)=0$ for every $f\in H(\D)$.

In the definition of the operators $V_g$ it can be assumed, if necessary, that $g(0)=0$. Indeed, the functions $g$ and $G(z):=g(z)-g(0)$, for $z\in\D$, define the same Volterra operator in $H(\D)$, that is, $V_G=V_g$, because $G'(z)=g'(z)$ for all $z\in\D$. For $g(z):=z$ the operator $V_g$ reduces to the classical (Volterra) integral operator.

 The Volterra operators $V_g$ have been investigated on different spaces of analytic functions by many
authors. We refer to \cite{Cont-a,Si}, for example,  and the references therein.

 Let us briefly recall the definition of the relevant spaces involved.
 For each $\gamma>0$ the \textit{Korenblum growth Banach spaces} are defined by
 \[
 A^{-\gamma}:=\{f\in H(\D):\ \|f\|_{-\gamma}:=\sup_{z\in\D}(1-|z|)^\gamma |f(z)|<\infty\}
 \]
and its (proper) closed subspace by
\[
 A^{-\gamma}_0:=\{f\in H(\D):\ \lim_{|z|\to 1^-}(1-|z|)^\gamma |f(z)|=0\}.
\]
Both are Banach  spaces  when endowed with the norm
\begin{equation}\label{eq.supnorm}
\|f\|_{-\gamma}:=\sup_{z\in\D}(1-|z|)^\gamma |f(z)|.
\end{equation}
The space  $A_0^{-\gamma}$ coincides with the closure of the polynomials in $A^{-\gamma}$, \cite[Lemma 3]{Sh-Wi}, and point evaluations on $\D$ belong to both $(A_0^{-\gamma})^*$ and $(A^{-\gamma})^*$, \cite[Lemma 1]{Sh-Wi}. Hence, via Lemma \ref{Lemma A}, both $A_0^{-\gamma}$ and $A^{-\gamma}$ are Banach spaces of analytic functions on $\D$, for every $\gamma>0$.
Moreover, the bidual $(A_0^{-\gamma})^{**}=A^{-\gamma}$ for all $\gamma>0$, \cite{Ru-Sh},
\cite[Theorem 2]{Sh-Wi}. Whenever $0<\gamma<\beta$, the inclusion $A^{-\gamma}\subseteq A^{-\beta}$ is \textit{proper}. This follows from the fact (routine to verify) that $A^{-\gamma}\subseteq A^{-\beta}_0$ and that $A^{-\beta}_0$ is a proper subspace of $A^{-\beta}$.

 Related to $A^{-\gamma}$ and $A^{-\gamma}_0$ are the \textit{Bloch spaces}. A function $f\in H(\D)$  belongs to the \textit{Bloch space} $\cB$ whenever
  $f'\in  A^{-1}$, that is, $\sup_{z\in\D}(1-|z|)|f'(z)|<\infty$. The Bloch space $\cB$ is a  Banach space of analytic  functions on $\D$ when endowed
 with the norm
 \begin{equation}\label{eq.normbloch}
 	\|f\|_{\cB}:=|f(0)|+\sup_{z\in\D}(1-|z|)|f'(z)|.
 	\end{equation}
 The inequalities
 \[
 (1-|z|)\leq (1-|z|^2)=(1-|z|)(1+|z|)\leq 2 (1-|z|),\quad z\in\D,
 \]
 show that the norm \eqref{eq.normbloch} is equivalent to the norm
 \[
 f\mapsto |f(0)|+\sup_{z\in\D}(1-|z|^2)|f'(z)|
 \]
 in $\cB$, which is also commonly used.
 A function $f\in H(\D)$ belongs to the \textit{little Bloch space} $\cB_0$ if $f\in \cB$ and
 \begin{equation}\label{eq.lim}
 	\lim_{|z|\to 1^-}(1-|z|)|f'(z)|=0.
 	\end{equation}
 The little Bloch space $\cB_0$  is a closed subspace of the Bloch space $\cB$ and  hence, $\cB_0$  is a  Banach space when it is endowed with the norm defined in  \eqref{eq.normbloch}. Moreover, the bidual $\cB_0^{**}=\cB$. For properties of Bloch spaces, see \cite{ACP}, \cite{Z}, for example.

 The following continuity and   compactness results for the operators $V_g$
 on both $A^{-\gamma}$ and $A_0^{-\gamma}$ are known; see \cite[Theorems 1 and 2]{BCHMP} or \cite[Proposition 3.1]{Ma}.

 \begin{prop}\label{P.Con} Let $g\in H(\D)$. For each $\gamma>0$  the  following three properties are equivalent.
 	\begin{itemize}
 		\item[\rm (i)] The operator  $V_g\colon A^{-\gamma}\to A^{-\gamma}$ is continuous.
 			\item[\rm (ii)] The operator $V_g\colon A^{-\gamma}_0\to A^{-\gamma}_0$ is continuous.
 				\item[\rm (iii)] The function $g\in \cB$.
 	\end{itemize}
 	\end{prop}
 	
 	\begin{prop}\label{P.Comp} Let $g\in H(\D)$. For each $\gamma>0$  the following three properties are equivalent.
 		\begin{itemize}
 			\item[\rm (i)] The operator $V_g\colon A^{-\gamma}\to A^{-\gamma}$ is compact.
 			\item[\rm (ii)] The operator $V_g\colon A^{-\gamma}_0\to A^{-\gamma}_0$ is compact.
 			\item[\rm (iii)] The function $g\in \cB_0$.
 		\end{itemize}
 	\end{prop}
  	
Since  $A^{-\gamma}$ and $A^{-\gamma}_0$ are Banach spaces of analytic functions on $\D$ and, for any non-constant function $g\in \cB$,  the Volterra operator $V_g\in \cL(H(\D))$ is injective and satisfies both  $V_g(A^{-\gamma})\subseteq A^{-\gamma}$ and $V_g(A_0^{-\gamma})\subseteq A_0^{-\gamma}$ (cf. Proposition \ref{P.Con}), we can define the optimal domain for each of the operators $V_g\colon A^{-\gamma}\to A^{-\gamma}$ and $V_g\colon A_0^{-\gamma}\to A^{.\gamma}_0$ via Definition \ref{Def-B}. Namely,
\[
[V_g, A^{-\gamma}]:=\{f\in H(\D):\ V_gf\in A^{-\gamma}\}
\]
and
\[
[V_g, A_0^{-\gamma}]:=\{f\in H(\D):\ V_gf\in A^{-\gamma}_0\},
\]
where both spaces are  endowed, respectively, with the semi-norm
\begin{equation}\label{eq.normD}
	\|f\|_{[V_g, A^{-\gamma}]}:=\|V_gf\|_{-\gamma}\ \ \mbox{and}\ \ \|f\|_{[V_g, A_0^{-\gamma}]}:=\|V_gf\|_{-\gamma}.
 	\end{equation}
 	Since $V_g\in \cL(H(\D))$ is injective, Proposition \ref{P1}(ii) implies that 	$\|\cdot\|_{[V_g, A^{-\gamma}]}$ and $\|\cdot\|_{[V_g, A_0^{-\gamma}]}$ are actually   norms. Moreover, again by Proposition \ref{P1}(v)  we have that $[V_g, A_0^{-\gamma}]$ is a closed subspace of $[V_g, A^{-\gamma}]$.
 	
 	We point out, for the Hardy spaces $H^p$ on $\D$, with $1<p<\infty$, that the study of the optimal domain of $V_g\colon H^p\to H^p$ has recently  been  treated in \cite{BDNS}.
 	
 	\begin{prop}\label{P.D1} Let $g\in \cB$ be a non-constant function. For each $\gamma>0$, both $[V_g,A^{-\gamma}]$ and $[V_g, A^{-\gamma}_0]$ are Banach spaces.
 		\end{prop}
 	
\begin{proof} In view of Proposition \ref{P3} to establish that $[V_g,A^{-\gamma}]$ is a Banach space,  it suffices to show that the inclusion map $J\colon [V_g, A^{-\gamma}]\to H(\D)$, defined by  $f\mapsto Jf:=f$, is continuous. To this effect, note that Theorem 5.5 in \cite{Du}, with $p=\infty$, implies that the differentiation operator $D\in \cL(H(\D))$ given by  $Dh:=h'$, for $h\in H(\D)$, maps $A^{-\gamma}$ into $A^{-(\gamma+1)}$. Since the evaluation functionals at points of $\D$ belong to $(A^{-\gamma})^*$ and $(A^{-(\gamma+1)})^*$, a closed graph argument shows that $D\in \cL(A^{-\gamma},A^{-(\gamma+1)})$; see also \cite[Theorem 2.1(a)]{HL}.
	Hence, there exists $C>0$ such that
	\begin{equation}\label{eq.Der}
		\|h'\|_{-(\gamma+1)}=\|Dh\|_{-(\gamma+1)}\leq C\|h\|_{-\gamma},\quad h\in A^{-\gamma}.
	\end{equation}
	Now, fix $r\in (0,1)$ and $f\in [V_g,A^{-\gamma}]$. Since $g'$ is not identically zero on $\D$ (as $g$ is a non-constant function) and the zeros of $g'\in H(\D)$ are isolated, there exists $s\in (r,1)$ such that $g'(z)\not=0$ for every $z\in \D$ such that $|z|=s$. Accordingly, $m:=\min_{|z|=s}|g'(z)|>0$.
	It follows, for every $u\in\D$ satisfying $|u|\leq r$, that
	\begin{equation}\label{eq,st}
	|f(u)|\leq \max_{|z|=s}|f(z)|=\max_{|z|=s}\frac{|f(z)g'(z)|}{|g'(z)|}\leq \frac{1}{m}\max_{|z|=s}|f(z)g'(z)|.
	\end{equation}
	On the other hand,
	$f(z)g'(z)=(V_gf)'(z)$, for $z\in\D$,
and hence, via \eqref{eq.Der}, we obtain that
	\begin{equation}\label{eq.Deri}
\|fg'\|_{-(\gamma+1)}=	\|(V_gf)'\|_{-(\gamma+1)}\leq C\|V_gf\|_{-\gamma}=C\|f\|_{[V_g, A^{-\gamma}]}.
\end{equation}
	Combining \eqref{eq,st} and \eqref{eq.Deri} it follows, for every $u\in\D$ with $|u|\leq r$, that
	\begin{align*}
		|f(u)|\leq & \frac{1}{m}\max_{|z|=s}|f(z)g'(z)|=\frac{1}{m}\max_{|z|=s}|(V_gf)'(z)|\\
		=&\frac{1}{m}\frac{1}{(1-s)^{\gamma+1}}(1-s)^{\gamma+1}\max_{|z|=s}|(V_gf)'(z)|\\
		=&\frac{1}{m}\frac{1}{(1-s)^{\gamma+1}}\max_{|z|=s}(1-|z|)^{\gamma+1}|(V_gf)'(z)|\\
		\leq & \frac{1}{m}\frac{1}{(1-s)^{\gamma+1}}\sup_{z\in\D}(1-|z|)^{\gamma+1}|(V_gf)'(z)|\\
		\leq & \frac{C}{m(1-s)^{\gamma+1}}\sup_{z\in\D}(1-|z|)^{\gamma}|(V_gf)(z)|\\
		=& \frac{C}{m(1-s)^{\gamma+1}}\|V_gf\|_{-\gamma}=\frac{C}{m(1-s)^{\gamma+1}}\|f\|_{[V_g,A^{-\gamma}]}.
	\end{align*}
In view of \eqref{eq.norme-sup} this implies that
\[
q_r(f)=\sup_{|u|\leq r}|f(u)|\leq \frac{C}{m(1-s)^{\gamma+1}}\|f\|_{[V_g,A^{-\gamma}]}.
\]
Since $r\in (0,1)$ and $f\in [V_g, A^{-\gamma}]$ are arbitrary, it follows that the inclusion $J\colon [V_g, A^{-\gamma}]\to H(\D)$ is continuous. As noted above, this implies that $[V_g,A^{-\gamma}]$ is a Banach space.

Recall that  $A^{-\gamma}_0$ is a closed subspace of $A^{-\gamma}$ and $V_g(A^{-\gamma}_0)\subseteq A^{-\gamma}_0$. So,   Proposition \ref{P1}(v) implies that $[V_g, A^{-\gamma}_0]$ is a closed subspace of  $[V_g, A^{-\gamma}]$. Accordingly, $[V_g, A^{-\gamma}_0]$ is also a Banach space.
\end{proof}

\begin{corollary}\label{CoroE}
Let $g\in\cB$ be a non-constant function. For each $\gamma>0$, both of the optimal domain spaces $[V_g,A^{-\gamma}]$ and  $[V_g,A_0^{-\gamma}]$ are Banach spaces of analytic functions on $\D$. In particular, their dual spaces contain $\{\delta_z:\ z\in\D\}$.
\end{corollary}

\begin{proof}
Proposition \ref{P.D1} shows that both $[V_g,A^{-\gamma}]$ and  $[V_g,A_0^{-\gamma}]$ are Banach spaces. Since $V_g\in \cL(H(\D))$ is injective, Proposition \ref{P3} gives the desired conclusion.
\end{proof}

The following result gives an alternate description of the optimal domain spaces $[V_g,A^{-\gamma}]$ and  $[V_g,A_0^{-\gamma}]$.

\begin{prop}\label{P.Description} Let $g\in \cB$ be a non-constant function and $\gamma>0$.
	\begin{itemize}
		\item[\rm (i)] The optimal domain space $[V_g,A^{-\gamma}]=\{f\in H(\D):\ fg'\in A^{-(\gamma+1)}\}$. Moreover, there exists $C>0$ satisfying
		  \[\|fg'\|_{-(\gamma+1)}\leq C\|f\|_{[V_g,A^{-\gamma}]},\quad  f\in [V_g,A^{-\gamma}].
		  \]
		\item[\rm (ii)] The optimal domain space $[V_g,A_0^{-\gamma}]=\{f\in H(\D):\ fg'\in A_0^{-(\gamma+1)}\}$. Moreover, there exists $C>0$ satisfying
		\[
		\|fg'\|_{-(\gamma+1)}\leq C\|f\|_{[V_g,A^{-\gamma}_0]},\quad f\in [V_g,A_0^{-\gamma}].
		\]
	\end{itemize}
	\end{prop}

 \begin{proof} (i) According to Definition \ref{Def-B} and \eqref{eq.V}, a function $f\in H(\D)$ belongs to $[V_g, A^{-\gamma}]$ if and only if the function $(V_gf)(z)=\int_0^zf(\xi)g'(\xi)\,d\xi$, for $z\in\D$, belongs to $A^{-\gamma}$. But, $(V_gf)'=fg'$ and so \cite[Theorem 5.3]{Du} implies that this is equivalent to the fact that $fg'\in A^{-(\gamma+1)}$; see also \cite[Proposition 2.2(a)]{HL}. This shows that $[V_g,A^{-\gamma}]=\{f\in H(\D):\ fg'\in A^{-(\gamma+1)}\}$ as sets.
 	
 	 Let $C>0$ be a constant satisying \eqref{eq.Deri} for every $f\in  [V_g, A^{-\gamma}]$. It follows that
 	\[
 	\|fg'\|_{-(\gamma+1)}=\|(V_gf)'\|_{-(\gamma+1)}\leq C\|V_gf\|_{-\gamma}=C\|f\|_{[V_g,A^{-\gamma}]},\quad f\in [V_g, A^{-\gamma}],
 	\]
 	which is the stated inequality.
 	
 	 (ii) Observe that a function $f\in A^{-\gamma}_0$ if and only if $f'\in A^{-(\gamma+1)}_0$ (see the proof of Fact 2 in  \cite[Theorem 3.2]{ABR-R}). So, the result follows by a similar  argument as in part (i).
 	\end{proof}

 \begin{remark}\label{R1}\rm  Let $\gamma>0$ and  $g\in \cB$ satisfy $|g'|>0$ on $\D$ and $\frac{1}{g'}\in H^\infty(\D)$.
 	
 	(i) Let the linear space $E:=\{f\in H(\D):\ fg'\in A^{-(\gamma+1)}\}$ be endowed with the norm
 	\[
 	\|f\|_E:=\|fg'\|_{-(\gamma+1)},\quad f\in E.
 	\]
 	The claim is that  $E$ is a Banach space isomorphic to $[V_g, A^{-\gamma}]$. Indeed, let $(f_n)_{n\in\N}\subseteq E$ be a Cauchy sequence in $E$, that is, $(f_ng')_{n\in\N}$ is a Cauchy sequence in $A^{-(\gamma+1)}$. Then there exists $h\in A^{-(\gamma+1)}$ such that $f_ng'\to h$ in $A^{-(\gamma+1)}$. But,  the multiplication operator $M_{\frac{1}{g'}}\colon A^{-(\gamma+1)}\to A^{-(\gamma+1)}$ is continuous (as $\frac{1}{g'}\in H^\infty(\D)$) and so  $f_n\to \frac{h}{g'}$ in $A^{-(\gamma+1)}$. Since $\frac{h}{g'}\in H(\D)$, it follows that $f_n\to \frac{h}{g'}$ in $E$. Hence, $E$ is a Banach space.
 	
 	Now, by Proposition \ref{P.Description}(i) the  operator $I\colon [V_g,A^{-\gamma}]\to E$, defined by $f\mapsto f$, is continuous. On the other hand, by Proposition \ref{P.D1}, $[V_g, A^{-\gamma}]$ is  a Banach space. So, by the open mapping theorem we can conclude that  $E$ is a Banach space which is isomorphic to $[V_g, A^{-\gamma}]$, that is, the norms $\|\cdot\|_{[V_g,A^{-\gamma}]}$ and $\|\cdot\|_E$ are equivalent.
 	
 	(ii) Let the linear space $E_0:=\{f\in H(\D):\ fg'\in A_0^{-(\gamma+1)}\}$ be  endowed with the norm $\|\cdot\|_E$. Clearly, $E_0$ is a subspace of $E$. Moreover, $E_0$ is a Banach space. Indeed, $E_0$ is a closed subspace of $E$. This follows after  observing that if $(f_n)_{n\in\N}\subseteq E_0$ is a sequence converging to some $f$ in $E$, then $(f_ng')_{n\in\N}\subseteq A_0^{-(\gamma+1)}$ and $f_ng'\to fg'$ in $A^{-(\gamma+1)}$. Hence, $fg'\in A^{-(\gamma+1)}_0$ as $A_0^{-(\gamma+1)}$ is a closed subspace of $A^{-(\gamma+1)}$.
 	
 	 As in part (i), it can be argued that $E_0$ is isomorphic to  $[V_g, A^{-\gamma}_0]$.
 \end{remark}

We apply Proposition \ref{P.Description}  to show that the optimal domain space of $V_g$ can be genuinely larger than $A^{-\gamma}$.

\begin{example}\label{E1}\rm  Let $g\in \cB$ satisfy $|g'|>0$ on $\D$ and $\frac{1}{g'}\in H^\infty$. For example, $g'(z)=\frac{1}{1-z}$, for $z\in\D$.  Suppose  there exists $w\in \C$ with $|w|=1$ such that $\lim_{r\to 1^-}|g'(rw)|(1-r)=0$. In this case the claim is that  $A^{-\gamma}$ is a \textit{proper} subspace of $[V_g,A^{-\gamma}]$, for all $\gamma>0$. For instance, consider  $g(z):=-{\rm Log}(1-z)$, for $z\in\D$, and  $w=-1$. Or, if $g\in \cB_0$ satisfies  $|g'|>0$ on $\D$ and $\frac{1}{g'}\in H^\infty$, then the same features occur because $\lim_{|z|\to 1^-}(1-z)|g'(z)|=0$ implies that $\lim_{r\to 1^-}|g'(rw)|(1-r)=0$.
	
To establish the claim let $f(z):=\frac{1}{g'(z)(1-\overline{w}z)^{\gamma+1}}$, for $z\in\D$. Since $1-\overline{w}z=0$ if and only if $z=w\not\in\D$, and $g'(z)\not=0$ for all $z\in\D$, we have that $f\in H(\D)$. On the other hand, $|1-\overline{w}z|\geq 1-|\overline{w}z|=1-|z|$ for all $z\in\D$, and hence, $\sup_{z\in\D}\frac{(1-|z|)^{\gamma+1}}{|1-\overline{w}z|^{\gamma+1}}\leq 1$. It follows that  the function
\[
f(z)g'(z)=\frac{1}{(1-\overline{w}z)^{\gamma+1}},\quad z\in\D,
\]
belongs to $A^{-(\gamma+1)}$. In view of Proposition \ref{P.Description}(i) we conclude that $f\in [V_g, A^{-\gamma}]$. But,
\begin{align*}
	 \sup_{z\in\D}(1-|z|)^\gamma|f(z)|& \geq \sup_{r\in (0,1)}(1-|rw|)^{\gamma}\frac{1}{|g'(rw)||1-\overline{w}rw|^{\gamma+1}} \\
	& =\sup_{r\in (0,1)}\frac{1}{|g'(rw)||1-r|}=\infty
\end{align*}
because $\frac{1}{g'}\in H^\infty$. Accordingly, $f\not\in A^{-\gamma}$.

Suppose, in addition, that  the function  $g\in \cB$  satisfies  $|g'(rw)|(1-r)^{1/2}\leq c$ for every $r\in (0,1)$ and some constant $c>0$. The claim is that also  $A_0^{-\gamma}$ is a \textit{proper} subspace of $[V_g,A_0^{-\gamma}]$. Indeed, let
$f_0(z):=(1-z)^{1/2}f(z)$, for $z\in\D$. Since $\lim_{|z|\to 1^-}(1-z)^{1/2}=0$ and $fg'\in A^{-(\gamma+1)}$,  the function
\[
f_0(z)g'(z)=(1-z)^{1/2}f(z)g'(z),\quad z\in\D,
\]
belongs to $A_0^{-(\gamma+1)}$. In view of Proposition \ref{P.Description}(ii) we can conclude that $f\in [V_g, A^{-\gamma}]$. But, $f\not\in A_0^{-\gamma}$ because, for every $r\in (0,1)$, we have
\begin{align*}
&(1-|rw|)^{\gamma}|f(rw)|\geq \frac{(1-|rw|)^{\gamma}(1-|rw|)^{1/2}}{|g'(rw)|1-\overline{w}rw|^{\gamma+1}}\\
&=\frac{(1-r)^{\gamma+1/2}}{|g'(rw)|(1-r)^{\gamma+1}}=\frac{1}{|g'(rw)|(1-r)^{1/2}}\geq \frac{1}{c},
\end{align*}
which implies that $\lim_{|z|\to 1^-}(1-|z|)^\gamma|f(z)\not=0$.
	\end{example}

We now formulate useful descriptions of the Banach spaces  $A^{-\gamma}$ and  $A^{-\gamma}_0$ in terms of the optimal domain of Volterra operators.

\begin{prop}\label{P.Intr} For every $\gamma>0$ the Korenblum growth space $A^{-\gamma}$ satisfies
	\begin{equation}
		A^{-\gamma}=\cap_{g\in \cB}[V_g,A^{-\gamma}].
	\end{equation}
	\end{prop}

\begin{proof}
	From Proposition \ref{P1}(iii) and Proposition \ref{P.Description} it follows that $A^{-\gamma}\subseteq \cap_{g\in \cB}[V_g,A^{-\gamma}]$.
	
	To establish the reverse inclusion,  fix $f\in \cap_{g\in \cB}[V_g,A^{-\gamma}]$ and consider the operator $S_f\colon H(\D)\to H(\D)$ defined by
	\[
	(S_fg)(z):=\int_0^zg'(\xi)f(\xi)\,d\xi=\int_0^zf(\xi)(Dg)(\xi)\,d\xi,\quad g\in H(\D),\ z\in\D.
	\]
	The facts that $f\in H(\D)$ and $D\in \cL(H(\D))$ clearly imply  that $S_f\in \cL(H(\D))$; see \eqref{eq.norme-sup} for the definition of $\tau_c$.
	On the other hand,  for each  $g\in\cB$, we have that  $V_gf\in A^{-\gamma}$ (cf. Proposition \ref{P1}(iii) and Proposition \ref{P.Description}) and
	\[
	(S_fg)(z)=\int_0^zg'(\xi)f(\xi)\,d\xi=(V_gf)(z),\quad  z\in\D.
	\]
	Therefore, $S_f(\cB)\subseteq A^{-\gamma}$ and hence, the linear map $S_f\colon \cB\to A^{-\gamma}$. Moreover, the operator  $S_f\colon \cB\to A^{-\gamma}$ has a closed graph and hence, it is continuous. To see this let $(g_n)_{n\in\N}\subseteq \cB$ be a sequence such that $g_n\to g$ in $\cB$ and $S_fg_n\to h$ in $A^{-\gamma}$ as $n\to\infty$. Since both $\cB$ and $A^{-\gamma}$ are Banach spaces of analytic functions on $\D$ we have that  $g_n\to g$  and $S_fg_n\to h$ in $H(\D)$ as $n\to\infty$. But,  $S_f\in \cL(H(\D))$ and so  $S_fg_n\to S_fg$ in $H(\D)$ as $n\to\infty$. Accordingly,  $S_fg=h$. So, $S_f\in \cL(\cB, A^{-\gamma})$.
	
	Recall from above that the differentiation  operator $D\colon A^{-\gamma}\to A^{-(\gamma+1)}$ is continuous and hence, also the operator $D\circ S_f\colon \cB\to A^{-(\gamma+1)}$, given by $g\mapsto g'f$, is continuous. Accordingly. there exists $C>0$ such that
	\begin{equation}\label{eq.DD}
		\|(D\circ S_f)g\|_{-(\gamma+1)}=\sup_{z\in\D}(1-|z|)^{\gamma+1}|g'(z)f(z)|\leq  C\|g\|_{\cB},\quad g\in\cB.
		\end{equation}
	In particular, for $g(z)=z$, for $z\in\D$,  it follows from \eqref{eq.DD} and $\|g\|_{\cB}=1$ that $|f(0)|\leq C$.
	
	Now, for each $u\in\D\setminus \{0\}$ let $g_u\in H(\D)$ satisfy $g'_u(z):=\frac{1}{1-\frac{\overline{u}}{|u|}z}$, for $z\in\D$, and $g_u(0)=0$ (namely, $g_u(z)=-\frac{|u|}{\overline{u}}{\rm Log} (1-\frac{\overline{u}}{|u|}z)$. Then, for each $u\in\D\setminus \{0\}$, we have that $\sup_{z\in\D}(1-|z|)|g'_u(z)|\leq 1$ and so $g_u\in \cB$.  Moreover, $(1-|u|)|g_u'(u)|=\frac{1-|u|}{|1-\frac{\overline{u}u}{|u|}|}=1$ for every  $u\in\D\setminus \{0\}$. Hence, $\|g_u\|_{\cB}=\sup_{z\in\D}(1-|z|)|g'_u(z)|=1$. Via \eqref{eq.normbloch} and \eqref{eq.DD} we can conclude that
	\[
	(1-|u|)^{\gamma+1}|g_u'(u)||f(u)|\leq C\sup_{z\in\D}(1-|z|)|g'_u(z)|= C,\quad u\in\D\setminus \{0\},
	\]
	 that is,
	\[
	\frac{(1-|u|)^{\gamma+1}|f(u)|}{|1-\frac{\overline{u}u}{|u|}|}=(1-|u|)^\gamma|f(u)|\leq C, \quad u\in\D\setminus \{0\}.
	\]
	It follows that $f\in A^{-\gamma}$.
\end{proof}

The analogue of Proposition \ref{P.Intr} for $A_0^{-\gamma}$ is as follows.

\begin{prop}\label{PropF}For each $\gamma>0$, the Korenblum growth space $A^{-\gamma}_0$ satisfies
	\[
	A_0^{-\gamma}=\cap_{g\in \cB}[V_g,A^{-\gamma}_0].
	\]
\end{prop}

\begin{proof} Proposition \ref{P1}(iii) and Proposition \ref{P.Description} imply that 	$A_0^{-\gamma}\subseteq \cap_{g\in \cB}[V_g,A^{-\gamma}_0]$.
	
	For the reverse inclusion fix $f\in\cap_{g\in \cB}[V_g,A^{-\gamma}_0]$. It follows from Proposition \ref{P.Intr}  (as $A^{-\gamma}_0\subseteq A^{-\gamma}$) that $f\in A^{-\gamma}$ and $f\in \cap_{g\in \cB}[V_g,A^{-\gamma}]$. Proposition 3.1 of \cite{Cont-Diaz}, for the choices $\varphi(z):=z$ and $\psi(z):=f(z)$, implies that the multiplication operator $M_f\colon A^{-1}\to A^{-(\gamma+1)}$, given by $h\mapsto fh$ for $h\in A^{-1}$, is continuous. Actually, $M_f(A^{-1})\subseteq A^{-(\gamma+1)}_0$. Indeed, given $h\in A^{-1}$, define $g\in H(\D)$ by $g(z):=\int_0^zh(\xi)\,d\xi$ for $z\in\D$. Then $g'=h$ and $g\in \cB$ because
	\[
	\sup_{z\in \D}(1-|z|)|g'(z)|=\|h\|_{-1}<\infty.
	\]
	Moreover, $f\in [V_g, A^{-\gamma}_0]$ implies, via Proposition \ref{P.Description}(ii), that $fg'\in A^{-(\gamma+1)}_0$, that is, $M_f(h)=fh=fg'\in A^{-(\gamma+1)}_0$. Accordingly, $M_f\colon A^{-1}\to A^{-(\gamma+1)}_0$ is continuous and hence, so is its restriction $M_f\colon A^{-1}_0\to A^{-(\gamma+1)}_0$ to the closed subspace $A^{-1}_0\subseteq A^{-1}$. Since the bidual Banach space $(A^{-\alpha}_0)^{**}=A^{-\alpha}$ for all $\alpha>0$, the bi-transpose operator of $M_f\colon A^{-1}_0\to A^{-(\gamma+1)}_0$ is precisely $M^{**}_f\colon A^{-\gamma}\to A^{-(\gamma+1)}$. But, it was shown above that $M_f^{**}((A^{-1}_0)^{**})=M_f(A^{-1})\subseteq A^{-(\gamma+1)}_0$ and hence, $M_f\colon A^{-1}_0\to A^{-(\gamma+1)}_0$ is a weakly compact operator; see \cite[\S 42, Proposition 2(2)]{24}, for example. According to \cite[Theorem 5.1]{Cont-Diaz} the operator $M_f\colon A^{-1}_0\to A^{-(\gamma+1)}_0$ is then also a compact operator. So, applying Corollary 4.5 of \cite{Cont-Diaz} with $\varphi(z):=z$ for $z\in\D$, we can conclude that
	\[
	\lim_{|z|\to 1^-}\frac{|f(z)|(1-|z|)^{\gamma+1}}{1-|z|}=0,
	\]
	that is, $\lim_{|z|\to 1^-}|f(z)|(1-|z|)^{\gamma}=0$. This means precisely that $f\in A^{-\gamma}_0$.
\end{proof}

We will require the following fact; see \cite[Lemma 11]{BDNS}.

\begin{lemma}\label{L} Let   $\delta\in (0,1)$ and $g\in A^{-\delta}$.
		 For each $\varepsilon\in (0,\min\{\delta,1-\delta\})$ there exists a function $\phi\in H(\D)$ such that $\phi g\in A^{-(\delta+\varepsilon)}\setminus A^{-\delta}$.
\end{lemma}

\begin{prop}\label{PInclusione} Let $g\in \cB$ and $0<\gamma<\beta$.
	\begin{itemize}
		\item[\rm (i)] $[V_g, A^{-\gamma}]$ is a proper subspace of $[V_g, A^{-\beta}]$.
		\item[\rm (ii)] $[V_g, A_0^{-\gamma}]$ is a proper subspace of $[V_g, A_0^{-\beta}]$.
	\end{itemize}
\end{prop}

\begin{proof} (i) Since  $\gamma<\beta$, we have that  $A^{-\gamma}\subseteq A^{-\beta}$. The containment $[V_g, A^{-\gamma}]\subseteq [V_g, A^{-\beta}]$ is then clear  from the definition.
	
	Suppose that $[V_g, A^{-\gamma}]=[V_g, A^{-\beta}]$. Given  $h\in H(\D)$ denote by $Z(h)$ the discrete set of zeros of $h$ repeated with their multiplicity. Consider the subset $X:=\{G\in \cB:\ Z(g')\subseteq Z(G')\}$ of $\cB$. For each $G\in X$ note that $\frac{G'}{g'}\in H(\D)$.
	
	Now, if $\beta\geq (\gamma+1)$, then $[V_g,A^{-\gamma}]\subseteq [V_g, A^{-(\gamma+\frac{1}{2})}]\subseteq [V_g, A^{-\beta}]$ and hence, $[V_g, A^{-(\gamma+\frac{1}{2})}]= [V_g, A^{-\beta}]$.
	So, we can suppose that $\gamma<\beta<(\gamma+1)$ in which case $\delta:=1-(\beta-\gamma)\in (0,1)$.
	
	The claim is  that the assumption $[V_g,A^{-\gamma}]=[V_g,A^{-\beta}]$ implies that every $G\in X$ satisfies $G'\in A^{-\delta}$. To prove the claim we proceed as follows.
	
	Fix $G\in X$. Given $h\in A^{-\beta}$ then function $\psi:=\frac{G'h}{g'}\in H(\D)$ satisfies
	\begin{equation}\label{eq.Vi}
		V_g(\psi)(z)=V_g\left(\frac{G'h}{g'}\right)(z)=\int_0^zg'(\xi)\frac{G'(\xi)}{g'(\xi)}h(\xi)\,d\xi=V_G(h)(z),\quad z\in\D.
	\end{equation}
	Since $G\in \cB$,  Proposition \ref{P.Con} implies that $V_Gh\in A^{-\beta}$. Accordingly, via \eqref{eq.Vi}, it follows that $V_g\psi\in A^{-\beta}$, that is, $\psi\in [V_g, A^{-\beta}]=[V_g, A^{-\gamma}]$, and hence, $V_g\psi\in A^{-\gamma}$. So, $V_Gh\in A^{-\gamma}$.
	
	Since $h\in A^{-\beta}$ is arbitrary, we can conclude that $V_G(h)\in A^{-\gamma}$ for all $h\in A^{-\beta}$. Therefore, the continuous linear operator $V_G\colon H(\D)\to H(\D)$ satisfies $V_G(A^{-\beta})\subseteq A^{-\gamma}$. By the closed graph theorem it follows that  $V_G\colon A^{-\beta}\to A^{-\gamma}$ is also continuous. By \cite[Theorem 1]{BCHMP} this is equivalent to the fact that $$\sup_{z\in\D}\frac{(1-|z|)^\gamma}{(1-|z|)^\beta}(1-|z|)|G'(z)|<\infty.$$
	 Accordingly, $G'\in A^{-\delta}$. The claim is thereby  proved.
	
	Now, fix $G\in X$ such that $G'\not=0$. By the above claim  $G'\in A^{-\delta}$ with $\delta\in (0,1)$. Hence, by Lemma \ref{L}(i), for any $\varepsilon\in (0,\min\{\delta,1-\delta\})$ there exists $\phi\in H(\D)$ such that $\phi G'\in A^{-(\delta+\varepsilon)}\setminus A^{-\delta}$. Therefore, the function
	\[
	H(z):=\int_0^z\phi(\xi)G'(\xi)\,d\xi,\quad z\in \D,
	\]
	belongs to $H(\D)$, has derivative  $H'=\phi G'$ and $H\in X$ (as $Z(g')\subseteq Z(G')\subseteq Z(H')$). But,  $H'\not\in A^{-\delta}$. A contradiction.
	
	(ii) Arguing as in the proof of part (i) we can suppose that $(\beta-\gamma)<1$ and prove that the assumption $[V_g,A_0^{-\gamma}]=[V_g,A_0^{-\beta}]$ implies that every $G\in X$ satisfies $G'\in A^{-\delta}$, where $\delta:=1-(\beta-\delta)\in (0,1)$. To this effect fix  $G\in X$. For every $h\in A_0^{-\beta}$ the identity
	\[
	V_g\psi=V_Gh\in A^{-\gamma}_0,
	\]
	where $\psi=\frac{G'h}{g'}$ (cf. \eqref{eq.Vi}) implies that $V_g\psi\in A^{-\beta}_0$ and hence, $V_g\psi\in A^{-\gamma}_0$.  Therefore, the continuous linear operator  $V_G\colon H(\D)\to H(\D)$ satisfies $V_G(A^{-\beta}_0)\subseteq A^{-\gamma}_0$. According to the closed graph theorem  $V_G\colon A^{-\beta}_0\to A^{-\gamma}_0$ is actually continuous. By \cite[Theorem 1]{BCHMP} this is equivalent to the fact that $\sup_{z\in\D}\frac{(1-|z|)^\gamma}{(1-|z|)^\beta}(1-|z|)|G'(z)|<\infty$. Hence, $G'\in A^{-\delta}$.
	
	At this point, the contradiction follows by proceeding as in the latter part of the proof of part (i).
\end{proof}

Our final result in this section shows that the Banach space $[V_g, A^{-\gamma}_0]$ is separable for suitable functions $g\in \cB$.

\begin{prop}\label{P.DensePol} Let $\gamma>0$ and  $g\in \cB$ satisfy both $|g'|>0$ on $\D$ and $\frac{1}{g'}\in H^\infty$. Then  the space $\cP$ of all polynomials on $\D$ is dense in $[V_g, A^{-\gamma}_0]$.
\end{prop}

\begin{proof}
	Remark \ref{R1}(ii) implies that
	\[
	[V_g, A^{-\gamma}_0]=\{f\in H(\D):\ fg'\in A^{-(\gamma+1)}_0\}
	\]
	and that $f\mapsto\|fg'\|_{-(\gamma+1)}$ is an equivalent norm in $[V_g, A^{-\gamma}_0]$. So, for a fixed $f\in [V_g, A^{-\gamma}_0]$, the function $fg'\in A^{-(\gamma+1)}_0$ and hence, there exists a sequence $(p_n)_{n\in\N}\subseteq \cP$  such that $\|p_n-fg'\|_{-(\gamma+1)}\to 0$ as $n\to\infty$; see the discussion after \eqref{eq.supnorm}. We now observe that $\frac{p_n}{g'}\in A^{-\gamma}_0$ for each $n\in\N$. Indeed, for $n\in\N$ fixed, we have
	\[
	(1-|z|)^\gamma |p_n(z)|\frac{1}{|g'(z)|}\leq (1-|z|)^\gamma |p_n(z)| \left\|\frac{1}{g'}\right\|_{H^\infty},\quad z\in\D.
	\]
	It follows that
	\[
	\lim_{|z|\to 1^-}(1-|z|)^\gamma |p_n(z)|\frac{1}{|g'(z)|}=0,
	\]
	that is, $\frac{p_n}{g'}\in A^{-\gamma}_0$.
	Since the space $\cP$ is dense in $A^{-\gamma}_0$ and $\left(\frac{p_n}{g'}\right)_{n\in\N}\subseteq A^{-\gamma}_0$, for every $n\in\N$ there exists $q_n\in \cP$ such that
	\[
	\left\|\frac{p_n}{g'}-q_n\right\|_{-\gamma}<\frac{1}{n}.
	\]
	Recall that $\sup_{z\in\D}(1-|z|)|g'(z)|\leq \|g\|_{\cB}$. So, for each $n\in\N$, it follows that
	\begin{align*}
		&\sup_{z\in\D}(1-|z|)^{\gamma+1}|q_n(z)-f(z)|\cdot |g'(z)|\\
		&\leq \sup_{z\in\D}(1-|z|)^{\gamma+1}\left|q_n(z)-\frac{p_n(z)}{g'(z)}\right||g'(z)|+ \sup_{z\in\D}(1-|z|)^{\gamma+1}\left|\frac{p_n(z)}{g'(z)}-f(z)\right||g'(z)|\\
		&\leq \|g\|_{\cB}\left\|\frac{p_n}{g'}-q_n\right\|_{-\gamma}+\|p_n-fg'\|_{-(\gamma+1)}\leq \frac{1}{n}\|g\|_{\cB}+\|p_n-fg'\|_{-(\gamma+1)}.
	\end{align*}
	This implies that
	\[
	\sup_{z\in\D}(1-|z|)^{\gamma+1}|q_n(z)-f(z)||g'(z)|\to 0
	\]
	as $n\to\infty$, that is, $q_n\to f$ in $[V_g, A^{-\gamma}_0]$. This completes the proof.
\end{proof}

\section{Multipliers of optimal domain spaces}

Let $X, Y$ be Banach spaces of analytic functions on $\D$. A function $h\in H(\D)$ is called a \textit{multiplier} for $X, Y$ if the multiplication operator $M_h\in \cL(H(\D))$ given by $f\mapsto hf$, for $f\in H(\D)$, satisfies $M_h(X)\subseteq Y$. We also say that  $M_h\colon X\to Y$ exists as a linear map. Since both $X\subseteq H(\D)$ and $Y\subseteq H(\D)$ continuously, a closed graph argument shows that necessarily $M_h\in \cL(X,Y)$. The space of all multipliers for $X, Y$ is denoted by $\cM(X,Y)$ or, if $X=Y$, simply by $\cM(X)$. 
The aim of this section is to identify the  multiplier spaces of the optimal domain spaces $[V_g,A^{-\gamma}]$ and $[V_g,A^{-\gamma}_0]$ for $g\in\cB$.

For each $\gamma>0$ the following operators are known to be  continuous.
\begin{itemize}
	\item[\rm (i)]The differentiation operator $D\colon A^{-\gamma}\to A^{-(\gamma+1)}$, defined by $h\mapsto Dh:=h'$; see Section 3.
	\item[\rm (ii)] The integration operator $J\colon A^{-(\gamma+1)}\to A^{-\gamma}$, defined by $h\mapsto (Jh)(z):=\int_0^zh(\xi)\,d\xi$, for $z\in\D$, which satisfies $\|J\|_{A^{-(\gamma+1)}\to A^{-\gamma}}\leq \frac{1}{\gamma}$; see \cite[pp. 236-237 \& Proposition 2.2(a)]{HL}, \cite[Proposition 3.12, Corollary 3.2]{AT}.
	\item[\rm (iii)] The multiplication operator $M_h\colon A^{-\gamma}\to A^{-\gamma}$, defined by $f\mapsto hf$, is continuous if and only if $f\in H^\infty$. In this case, $\|M_h\|_{A^{-\gamma}\to A^{-\gamma}}=\|h\|_\infty$, \cite[Proposition 2.1]{BDL}.
\end{itemize}

\begin{remark}\label{R.NormaInt} \rm The fact in (ii) above  that $\|J\|_{A^{-(\gamma+1)}\to A^{-\gamma}}\leq \frac{1}{\gamma}$ can be established as follows. Fix $f\in A^{-(\gamma+1)}$ and observe, for each $z\in \D$, that
\begin{align*}
	(1-|z|)^\gamma\left|\int_0^zf(\xi)\,d\xi\right|&=	(1-|z|)^\gamma\left|\int_0^1zf(tz)\,dt\right|\leq 	(1-|z|)^\gamma\int_0^1|z|\cdot|f(tz)|\,dt\\
	&=(1-|z|)^\gamma\int_0^1 \frac{|z|}{(1-t|z|)^{\gamma+1}}(1-t|z|)^{\gamma+1}|f(tz)|\,dt\\
	&\leq \|f\|_{-(\gamma+1)}(1-|z|)^\gamma \left[\frac{1}{\gamma(1-t|z|)^\gamma}\right]_{t=0}^{t=1}\\
	&=\|f\|_{-(\gamma+1)}(1-|z|)^\gamma\left(\frac{1}{\gamma(1-|z|)^\gamma}-\frac{1}{\gamma}\right)\\
	& =\|f\|_{-(\gamma+1)}\left(\frac{1}{\gamma}-\frac{(1-|z|)^\gamma}{\gamma}\right)\leq \frac{1}{\gamma}\|f\|_{-(\gamma+1)}.
\end{align*}
\end{remark}

The statements (i)--(iii) above also hold when these operators act on $A^{-\gamma}_0$.

The next result characterizes the action of multiplication operators on optimal domain spaces.

\begin{prop}\label{P.Mult} Let $g\in \cB$ and $h\in H(\D)$. For each $\gamma>0$ the following three properties are equivalent.
	\begin{itemize}
		\item[\rm (i)] The multiplication operator  $M_h\colon [V_g,A^{-\gamma}]\to [V_g,A^{-\gamma}]$ exists as a linear map and is continuous.
		\item[\rm (ii)] The multiplication operator  $M_h\colon [V_g,A^{-\gamma}_0]\to [V_g,A^{-\gamma}_0]$ exists as a linear map and is continuous.
		\item[\rm (iii)] The function $h\in H^\infty$.
	\end{itemize}
 In this case,  the norms $\|M_h\|_{[V_g, A^{-\gamma}]\to [V_g,A^{-\gamma}]}$ and $\|h\|_\infty$ are equivalent.
	\end{prop}

\begin{proof} (iii)$\Rightarrow$(i) Let $f\in [V_g,A^{-\gamma}]$ be fixed. Then $V_gf\in A^{-\gamma}$. The claim is that $V_g(hf)\in A^{-\gamma}$. To establish this claim, observe that $V_gf\in A^{-\gamma}$ implies that $(V_gf)'=g'f\in A^{-(\gamma+1)}$ (as the operator $D$ maps  $A^{-\gamma}$ into $A^{-(\gamma+1)}$). Since $h\in H^\infty$, the operator $M_h\in \cL(A^{-\gamma})$ and so  $hg'f\in A^{-(\gamma+1)}$.  Hence, $V_g(hf)=J(hg'f)\in A^{-\gamma}$, as the operator $J$ maps $A^{-(\gamma+1)}$ into $A^{-\gamma}$. The claim is thereby established.
	
From the previous paragraph we can conclude that $M_hf\in [V_g,A^{-\gamma}]$. Moreover,
	\begin{align*}
		&\|M_hf\|_{[V_g,A^{-\gamma}]}=\|V_g(hf)\|_{-\gamma}=\|(JM_hDV_g)(f)\|_{-\gamma}\\
		&\leq \|J\|_{A^{-(\gamma+1)}\to A^{-\gamma}}\|M_h\|_{A^{-(\gamma+1)}\to A^{-(\gamma+1)}}\|D\|_{A^{-\gamma}\to A^{-(\gamma+1)}}\|V_gf\|_{-\gamma}\\
		&\leq \frac{1}{\gamma}\|D\|_{A^{-\gamma}\to A^{-(\gamma+1)}}\|h\|_\infty\|f\|_{[V_g,A^{-\gamma}]}.
		\end{align*}
	This means that $M_h\colon [V_g,A^{-\gamma}]\to [V_g,A^{-\gamma}]$ is continuous with $\|M_h\|_{[V_g, A^{-\gamma}]\to [V_g,A^{-\gamma}]}\leq \frac{1}{\gamma} \|D\|_{A^{-\gamma}\to A^{-(\gamma+1)}}\|h\|_\infty$.
	
	(i)$\Rightarrow$(iii) Let $f\in [V_g,A^{-\gamma}]$ satisfy $\|f\|_{[V_g,A^{-\gamma}]}=1$. For fixed  $z\in\D$ observe, as $[V_g,A^{-\gamma}]$ is a Banach space of analytic functions on $\D$ (cf. Corollary \ref{CoroE}), that
	\[
	|h(z)f(z)|=|\langle M_h f,\delta_z\rangle|\leq \|M_h\|_{[V_g,A^{-\gamma}]\to [V_g,A^{-\gamma}]}\|\delta_z\|_{[V_g,A^{-\gamma}]^*}.
	\]
	Taking the supremum over all such  functions $f$ yields
	\[
	|h(z)|\cdot \|\delta_z\|_{[V_g,A^{-\gamma}]^*}\leq \|M_h\|_{[V_g,A^{-\gamma}]\to [V_g,A^{-\gamma}]}\|\delta_z\|_{[V_g,A^{-\gamma}]^*}.
	\]
	Since $z\in\D$ is arbitrary, it follows that $h\in H^\infty$ and $\|h\|_\infty\leq \|M_h\|_{[V_g,A^{-\gamma}]\to [V_g,A^{-\gamma}]}$.
	
	(iii)$\Leftrightarrow$(ii) Since each of the linear maps $D\colon A_0^{-\gamma}\to A_0^{-(\gamma+1)}$, $J\colon A_0^{-(\gamma+1)}\to A_0^{-\gamma}$ and $M_h\colon A_0^{-\gamma}\to A_0^{-\gamma}$ is defined and continuous,  the result follows by arguing as in the proof of (iii)$\Leftrightarrow$(i).
\end{proof}

As an immediate consequence we have the following result.

\begin{corollary}\label{coroG} For each $g\in \cB$ and each $\gamma>0$ the multiplier spaces of the optimal domain spaces $[V_g, A^{-\gamma}]$ and $[V_g, A^{-\gamma}_0]$ are given by
	\[
	\cM([V_g, A^{-\gamma}])=\cM([V_g, A_0^{-\gamma}])=H^\infty.
	\]
\end{corollary}

Certain multiplier spaces for the Korenblum spaces are known. The following result is Proposition 5 in \cite{BDNS}; see also \cite[Proposition 3.1]{Cont-Diaz}.

\begin{prop}\label{PH} Let $0<\gamma<\delta$ and $h\in H(\D)$. The multiplication operator $M_h\colon A^{-\gamma}\to A^{-\delta}$ exists if and only if $h\in A^{-(\delta-\gamma)}$. That is,
	\[
	\cM(A^{-\gamma},A^{-\delta})=A^{-(\delta-\gamma)}.
	\]
\end{prop}

Two further results in this direction are the following Propositions \ref{PH=} and \ref{L_M}.

\begin{prop}\label{PH=} Let $0<\gamma<\delta$ and $h\in H(\D)$. The multiplication operator $M_h\colon A^{-\gamma}_0\to A^{-\delta}_0$ exists if and only if $h\in A^{-(\delta-\gamma)}$. That is,
	\[
	\cM(A^{-\gamma}_0,A^{-\delta}_0)=A^{-(\delta-\gamma)}.
	\]
\end{prop}

\begin{proof}
	Suppose that the operator $M_h\colon A^{-\gamma}_0\to A^{-\delta}_0$ exists and is continuous. Then its bi--transpose operator $M^{**}_h\colon (A^{-\gamma}_0)^{**}\to (A^{-\delta}_0)^{**}$ is continuous. Since $(A^{-\gamma}_0)^{**}=A^{-\gamma}$ and   $(A^{-\delta}_0)^{**}=A^{-\delta}$ (and so $M^{**}_h=M_h$),  Proposition \ref{PH} implies that $h\in A^{-(\delta-\gamma)}$.
	
	Conversely, suppose that $h\in A^{-(\delta-\gamma)}$. Fix $f\in A^{-\gamma}_0$. For each $z\in\D$, we have that
	\[
	|(M_hf)(z)|(1-|z|)^\delta=|h(z)|(1-|z|)^{\delta-\gamma}|f(z)|(1-|z|)^\gamma\leq \|h\|_{-(\delta-\gamma)}|f(z)|(1-|z|)^\gamma.
	\]
	Since $f\in A^{-\gamma}$, this implies that $\lim_{|z|\to 1^-}|(M_hf)(z)|(1-|z|)^\delta=0$ and hence, $M_h\in A^{-\delta}_0$. So, $M_h\colon A^{-\gamma}_0\to A^{-\delta}_0$ exists as a linear map. The continuity follows from the closed graph theorem.
\end{proof}

A more general result than Proposition \ref{PH=}, concerning weighted composition operators, occurs in \cite[Proposition 3.2]{Cont-Diaz}.

An analogue of the above characterization is the following result.

\begin{prop}\label{L_M} Let $0<\gamma<\delta$ and $h\in H(\D)$.
The multiplication operator $M_h\colon A^{-\gamma}\to A^{-\delta}_0$ exists  if and only if $h\in A_0^{-(\delta-\gamma)}$. That is,
\[
\cM(A^{-\gamma},A^{-\delta}_0)=A^{-(\delta-\gamma)}_0.
\]
	\end{prop}

\begin{proof} Suppose that the operator $M_h\colon A^{-\gamma}\to A^{-\delta}_0$ exists and is continuous. According to \cite[Theorem 1.2]{AD} there exist functions $f_1,f_2\in H(\D)$ and positive constants $c_1$, $c_2$ such that $c_1(1-|z|)^{-\gamma}\leq |f_1(z)|+|f_2(z)|\leq c_2(1-|z|)^{-\gamma}$ for all $z\in \D$. Then both $f_1, f_2\in A^{-\gamma}$ and hence,  $M_hf_1, M_hf_2\in A^{-\delta}_0$. Moreover,
	\[
	|h(z)|(1-|z|)^{\delta-\gamma}\leq c_1^{-1} (1-|z|)^{\delta}(|f_1(z)|+|f_2(z)|)|h(z)|, \quad z\in\D.
	\]
	Since both  $M_hf_1, M_hf_2\in A^{-\delta}_0$, this implies that $\lim_{|z|\to 1^-}|h(z)|(1-|z|)^{\delta-\gamma}=0$, that is, $h\in A_0^{-(\delta-\gamma)}$.
	
	Conversely, suppose that $h\in A_0^{-(\delta-\gamma)}$. Fix  $f\in A^{-\gamma}$. For each  $z\in\D$, we have that
	\[
	|(M_hf)(z)|(1-|z|)^\delta=|h(z)|(1-|z|)^{\delta-\gamma}|f(z)|(1-|z|)^{\gamma}\leq \|f\|_{-\gamma}|h(z)|(1-|z|)^{\delta-\gamma}.
	\]
Since $h\in A^{-(\delta-\gamma)}$, this implies that $\lim_{|z|\to 1^-}|(M_hf)(z)|(1-|z|)^\delta=0$ and hence, that $M_hf\in A^{-\delta}_0$. So, $M_h\colon A^{-\gamma}\to A^{-\delta}_0$ exists as a linear map.  The continuity now follows from the closed graph theorem.
	\end{proof}

\section{The  range of $V_g$ and related properties}

	Let $X$ be a Banach space of analytic functions on $\D$. Whenever $T\colon X\to X$ is an injective continuous operator such that $([T,X],\|\cdot\|_{[T,X]})$ is a Banach space,
Proposition \ref{P5} and Corollary \ref{C1} reveal an intimate connection between the three properties of the operator $T$ having closed range, of the natural inclusion map $X\subseteq [T,X]$ having closed range, and of $X$ being a proper subspace of $[T,X]$. In the latter part of this section we investigate these connections for the Cesàro operator $C$ and for the operators $V_g$, when they act in Korenblum growth spaces. This requires various preliminary results concerning the auxiliary operators $S$ and $T$ defined in \eqref{5A} and \eqref{5B}, respectively, which act in certain weighted Banach spaces of analytic functions on $\D$ which we now introduce.

A \textit{weight} $v$ is a continuous, non-increasing function $v \colon [0, 1) \to (0,\infty)$. We
extend $v$ to $\D$ by setting $v(z) := v(|z|)$, for $z \in\D$. Note that $v(z) \leq v(0)$ for all $z\in\D$. It is assumed throughout this section that $\lim_{r\to 1^-}v(r)=0$.
 Given a weight $v$ on $[0, 1)$,  define the corresponding \textit{weighted Banach
spaces of analytic functions} on $\D$ by
\[
H^\infty_v := \{f\in H(\D) : \ \|f\|_{\infty,v} := \sup_{z\in\D}
	|f(z)|v(z) <\infty\},
	\]
and
\[
H^0_v := \{f \in H(\D) :\ \lim_{|z|\to 1^-}|f(z)|v(z) = 0\},
\]
both endowed with the norm $\|\cdot\|_{\infty,v}$.  For details concerning these spaces, see \cite{Bo} and the extensive list of references given there.

Observe that the Korenblum growth  spaces $A^{-\gamma}$ and $A^{-\gamma}_0$, for $\gamma>0$, correspond to $H^\infty_v$ and $H^0_v$, respectively, for the weight function $v(r):=(1-r)^\gamma$, for $r\in [0,1)$.

Consider now the forward shift operator $S\colon H(\D) \to H(\D)$ defined to be the multiplication operator
\begin{equation}\label{5A}
S(f)(z) := zf(z),\quad f \in H(\D), z \in\D.
\end{equation}
Clearly, $S\in \cL(H(\D))$.
Let the linear subspace $H_0(\D):=\{f\in H(\D):\ f(0)=0\}\subseteq H(\D)$ be equipped with the topology induced by $H(\D)$. Clearly, $H_0(\D)$ is a closed subspace of $H(\D)$ and $S(H(\D))\subseteq H_0(\D)$. Therefore, $S\in\cL( H(\D), H_0(\D))$. Moreover, $S\colon  H(\D)\to H_0(\D)$ is surjective. To see this, we consider the operator $T\colon H_0(\D) \to H(\D)$ defined, for each $f\in H_0(\D)$,  by $(Tf)(0):=f'(0)$ and
\begin{equation}\label{5B}
(Tf)(z):=\frac{1}{z}f(z),\quad z\in\D\setminus\{0\}.
\end{equation}
Since $(Tf)(z):=\frac{f(z)-f(0)}{z-0}$, for $z\in \D\setminus\{0\}$, it is clear that  $Tf\in H(\D)$ for all $f\in H_0(\D)$. Moreover, $S\in\cL( H_0(\D), H(\D))$. Indeed, for $r\in (0,1)$ and $f\in H_0(\D)$ fixed, note that
\[
\sup_{|z|\leq r}|(Tf)(z)|=\max_{|z|=r}|(Tf)(z)|=\frac{1}{r}\max_{|z|=r}|f(z)|.
\]
Now, given $h\in H_0(\D)$, recall  that $Th\in H(\D)$ and $(STh)(z)=z\frac{1}{z}h(z)=h(z)$ for all $z\in\D\setminus\{0\}$. So, $STh$ and $h$ are analytic functions on $\D$ which coincide except possibly at $0$. Therefore, $STh=h$ on $\D$. This proves that $S\colon  H(\D)\to H_0(\D)$ is surjective.

Since the evaluation functional $\delta_0$ belongs to both $(H_v^\infty)^{*}$ and $(H_v^0)^{*}$ we can define the Banach spaces
\[
\cE_v^\infty:=\{f\in H^\infty_v:\ f(0)=0\}=\Ker(\delta_0)
\]
and
\[
\cE_v^0:=\{f\in H^0_v:\ f(0)=0\}=\Ker(\delta_0),
\]
both endowed with the norm $\|\cdot\|_{\infty,v}$. Of course, $\cE_v^\infty$ (resp., $\cE_v^0$) is a closed subspace of $H^\infty_v$ (resp., $H^0_v$).

In the notation of Proposition \ref{P.DensePol} we have the following fact.

\begin{lemma}\label{L.Dens} Let $\cP_0:=\{p\in \cP:\  p(0)=0\}$.
	 Then the closure of $\cP_0$ in $H^\infty_v$ is equal to $\cE^0_v$.
\end{lemma}

\begin{proof} Let $g\in H^\infty_v$. If there exists $(p_n)_{n\in\N}\subseteq \cP_0$ such that $p_n\to g$ in $H^\infty_v$ as $n\to\infty$, then necessarily $g\in H^0_v$ because $\cP_0\subseteq H^0_v$ and $H_v^0$ is closed in $H^\infty_v$. Moreover, $g(0)=\lim_{n\to\infty}p_n(0)=0$. Accordingly, $g\in \cE^\infty_v$. Since  $g\in H^0_v$, we can conclude that $g\in \cE^0_v$.
	
	Let $h\in \cE_v^0\subseteq H^0_v$ be fixed. Then there exists a sequence  $(q_n)_{n\in\N}$ of polynomials such that $q_n\to h$ in $H^0_v$, \cite{BS}. Accordingly, $\lim_{n\to\infty}q_n(0)=h(0)=0$. This implies that $p_n:=q_n-q_n(0)\to h$ in $H^0_v$ as $n\to\infty$. Since  $(p_n)_{n\in\N}\subseteq \cP_0$, this completes the proof.
	\end{proof}

\begin{prop}\label{P_C1} Let $S$ and $T$ be the operators given by \eqref{5A} and \eqref{5B}, respectively.
	\begin{itemize}
		\item[\rm (i)] The operator $S$ satisfies  $S\in \cL(H^\infty_v)$ with $S(H^\infty_v)\subseteq \cE^\infty_v$ and $T$ satisfies $T\in \cL(\cE^\infty_v,H^\infty_v)$. Moreover, $STh=h$, for all $h\in \cE^\infty_v$.
			\item[\rm (ii)] The operator $S$ satisfies  $S\in \cL(H^0_v)$ with $S(H^0_v)\subseteq \cE^0_v$ and $T$ satisfies $T\in \cL(\cE^0_v,H^0_v)$. Moreover,   $STh=h$, for all $h\in \cE^0_v$.
	\end{itemize}
	\end{prop}

\begin{proof} (i) Let $f\in H^\infty_v$. Then
	\[
	\|Sf\|_{\infty,v}=\sup_{z\in\D}v(z)|z|\cdot |f(z)|\leq \sup_{z\in\D}v(z)|f(z)|=\|f\|_{\infty,v},
	\]
	which implies that $S\in \cL(H^\infty_v)$ with $\|S\|_{H^\infty_v\to H^\infty_v}\leq 1$. Clearly, $S(H^\infty_v)\subseteq \cE^\infty_v$.
	
Next, since $\cE^\infty_v\subseteq H_0(\D)$, observe that $T(\cE^\infty_v)\subseteq H(\D)$. Moreover, given  $f\in \cE^\infty_v$,  it follows from \eqref{5B} that
	\begin{align*}
	\sup_{|z|\leq 1/2}v(z)|(Tf)(z)|&\leq v(0)\max_{|z|=1/2}|(Tf)(z)|=v(0)\max_{|z|=1/2}\frac{1}{|z|}|f(z)|\\
&= \frac{2v(0)}{v(1/2)}\max_{|z|=1/2}v(1/2)|f(z)|\leq \frac{2v(0)}{v(1/2)}\|f\|_{\infty,v},
	\end{align*}
and that
\[
\sup_{1/2<|z|<1}v(z)|(Tf)(z)|=\sup_{1/2<|z|<1}v(z)\frac{1}{|z|}|f(z)|\leq 2\sup_{1/2<|z|<1}v(z)|f(z)|\leq 2\|f\|_{\infty,v}.
\]
Therefore,
\[
\|Tf\|_{\infty,v}\leq \max\left\{2,\frac{2v(0)}{v(1/2)}\right\}\|f\|_{\infty,v}.
\]
This implies that $T\in \cL(\cE^\infty_v,H^\infty_v)$.

Finally, for every $h\in \cE^\infty_v$ we have that $Th\in H^\infty_v$ and that $STh=h$; see the discussion after \eqref{5B}.

(ii) By part (i) we have that $S\in \cL(H^\infty_v)$. To conclude that $S\in \cL(H^0_v)$, it therefore suffices to establish that $S(H^0_v)\subseteq H^0_v$. So, fix  $f\in H^0_v$, in which case $\lim_{|z|\to 1^-}v(z)|f(z)|=0$. Since $v(z)|(Sf)(z)|\leq v(z)|f(z)|$ for all $z\in\D$, it follows that $\lim_{|z|\to 1^-}v(z)|(Sf)(z)|=0$. Accordingly, $Sf\in H^0_v$.

 Since $(Sf)(0)=0$ for all $f\in H_v^0$, it is clear that   $S(H^0_v)\subseteq \cE^0_v$.

 By part (i)
 we also  have that $T\in \cL(\cE^\infty_v,H^\infty_v)$. So, to complete the proof it remains to show  that $T(\cE^0_v)\subseteq H^0_v$. To see this,  observe that $Tp$ is a polynomial for every $p\in \cP_0$; see \eqref{5B}. Then $T(\cP_0)\subseteq H^0_v$ (as $\lim_{r\to 1^-}v(r)=0$). Since  $T\in \cL(\cE^\infty_v,H^\infty_v)$ and $H^0_v$ is closed in $H^\infty_v$, it follows via Lemma \ref{L.Dens} that
  \[
 T(\cE_v^0)=T(\overline{\cP_0}^{H^\infty_v})\subseteq \overline{T(\cP_0)}^{H^\infty_v}\subseteq \overline{H^0_v}^{H^\infty_v}=H^0_v.
 \]
So, $T\in \cL(\cE_v^0,H^0_v)$. Finally, $STh=h$ for each $h\in \cE^0_v$, as follows from part (i).
	\end{proof}

Related to the Volterra operator $V_g\colon H(\D)\to H(\D)$, where $g\in H(\D)$, is the operator $T_g\colon H(\D)\to H(\D)$ defined by
\begin{equation}\label{eq.T}
	(T_gf)(z):=\frac{1}{z}\int_0^zf(\xi)g'(\xi)\,d\xi,\ z\in \D\setminus\{0\},\quad (T_gf)(0):=f(0),
	\end{equation}
for each $f\in H(\D)$. Clearly, $T_g\in \cL(H(\D))$. Suppose that $T_gf=0$ for some $f\in H(\D)$. Then \eqref{eq.T} implies that $\int_0^zf(\xi)g'(\xi)\,d\xi=0$, that is, $(V_gf)(z)=0$ for all $z\in\D\setminus\{0\}$. Since also $(V_gf)(0)=0$, it follows that $V_gf=0$ in $H(\D)$. By the injectivity of $V_g\in \cL(H(\D))$ we can conclude that $f=0$. Accordingly, $T_g\in \cL(H(\D))$ is \textit{injective} and so its optimal domain spaces $[T_g, H^\infty_v]$ and $[T_g, H^0_v]$ are \textit{normed spaces} for every weight $v$ such that $T_g(H^\infty_v)\subseteq H^\infty_v$ (resp. $T_g(H^0_v)\subseteq H^0_v$); see Proposition \ref{P1}(ii) with $T=T_g$ and $X=H^\infty_v$ (resp. $H^0_v$).

\begin{prop}\label{P.Ugu} Let $v\colon [0,1)\to (0,\infty)$ be a weight function and $g\in H(\D)$.
	\begin{itemize}
		\item[\rm (i)] The operator  $V_g\colon H^\infty_v\to H^\infty_v$ is continuous if and only if the operator $T_g\colon H^\infty_v\to H^\infty_v$ is continuous.
		\item[\rm (ii)] Let $V_g\colon H^\infty_v\to H^\infty_v$ be continuous. Then $[V_g,H^\infty_v]=[T_g, H^\infty_v]$ as linear spaces in $H(\D)$ and with equivalent norms. 
		\item[\rm (iii)] The operator $V_g\colon H^0_v\to H^0_v$ is continuous if and only if the operator $T_g\colon H^0_v\to H^0_v$ is continuous.
		\item[\rm (iv)] Let $V_g\colon H^0_v\to H^0_v$ be continuous. Then $[V_g,H^0_v]=[T_g, H^0_v]$ as linear spaces in $H(\D)$ and with equivalent norms. 
	\end{itemize}
	\end{prop}

\begin{proof}
	(i) Suppose that $T_g\colon H^\infty_v\to H^\infty_v$ is continuous. By Proposition \ref{P_C1}(i) also $V_g=S\circ T_g\colon H^\infty_v\to \cE^\infty_v\subseteq H^\infty_v$ is continuous.
	
	Conversely, suppose that  $V_g\colon H^\infty_v\to H^\infty_v$ is continuous. Clearly, $V_g(H^\infty_v)\subseteq \cE^\infty_v$ and so  Proposition \ref{P_C1}(i) implies that  $T_g=T\circ V_g\colon H^\infty_v\to H^\infty_v$ is continuous.
	
	(ii) The assumption on $V_g$ ensures that  $T_g\colon H^\infty_v\to H^\infty_v$ is also continuous; see part (i).
	
	 Let $h\in [V_g, H^\infty_v]$. Then $h\in H(\D)$ and $V_gh\in H^\infty_v$. Since $(V_gh)(0)=0$, it follows that actually $V_gh\in \cE^\infty_v$. So, $h\in H(\D)$ and $T_gh=(T\circ V_g)h\in H^\infty_v$, after recalling that $T\in \cL(\cE^\infty_v, H^\infty_v)$ by Proposition \ref{P_C1}(i). Moreover,
	 \begin{align*}
	 \|h\|_{[T_g,H^\infty_v]}&:=\|T_gh\|_{\infty,v}=\|(T\circ V_g)h\|_{\infty,v}\leq \|T\|_{\cE^\infty_v\to H^\infty_v}\|V_gh\|_{\infty,v}\\
	 &= \|T\|_{\cE^\infty_v\to H^\infty_v}\|h\|_{[V_g,H^\infty_v]}<\infty,
	 \end{align*}
 that is, $h\in [T_g,H^\infty_v]$. Accordingly, $[V_g,H^\infty_v]\subseteq [T_g,H^\infty_v]$ with a continuous inclusion.

	 Let $f\in [T_g, H^\infty_v]$. Then $f\in H(\D)$ and $T_gf\in H^\infty_v$. By Proposition \ref{P_C1}(i) it follows that $V_gf=(S\circ T_g)f\in H^\infty_v$. Accordingly, $f\in [V_g,H^\infty_v]$. Moreover,
	 \begin{align*}
	 \|f\|_{[V_g,H^\infty_v]}&=\|V_gf\|_{\infty,v}=\|(S\circ T_g)f\|_{\infty,v}\leq \|S\|_{H^\infty_v\to \cE^\infty_v}\|T_gf\|_{\infty,v}\\
	 &= \|S\|_{H^\infty_v\to \cE^\infty_v}\|f\|_{[T_g,H^\infty_v]}<\infty,
	 \end{align*}
 that is, $f\in [V_g, H^\infty_v]$. Accordingly, $[T_g, H^\infty_v]\subseteq [V_g, H^\infty_v]$ with a continuous inclusion.

It follows that $[T_g, H^\infty_v]=[V_g, H^\infty_v]$	 as linear spaces of $H(\D)$ and that the norms $\|\cdot\|_{[T_g,H^\infty_v]}$ and $\|\cdot\|_{[V_g,H^\infty_v]}$ are \textit{equivalent}.

	In view of Proposition \ref{P_C1}(ii),  the proofs of parts (iii) and (iv) are similar.
\end{proof}

Let us consider a particular case. Define $g_0(z):=-{\rm Log}(1-z)$, for $z\in\D$. Then, $g_0'(z)=\frac{1}{1-z}$, for $z\in\D$, and  hence, $g\in \cB$. Moreover, for each $f\in H(\D)$, the operator $T_{g_0}$ is given by $(T_{g_0}f)(0)=f(0):=(Cf)(0)$ and
\begin{equation}\label{eq.CesaroO}
(T_{g_0})f(z)=\frac{1}{z}\int_0^z\frac{f(\xi)}{1-\xi}\,d\xi=:(Cf)(z),\ z\in\D\setminus\{0\},
\end{equation}
that is, $T_{g_0}=C$ is the classical  \textit{Cesàro operator}. We also point out that
\begin{equation}\label{eq.Volt_C}
(V_{g_0}f)(z)=\int_0^z\frac{f(\xi)}{1-\xi}\,d\xi,\quad z\in\D.
\end{equation}
Aleman and Persson have made an extensive investigation of
various properties of generalized Ces\`aro operators
acting in a large class of Banach spaces of analytic functions on $\D$, \cite{AP,AleSi1,P}.
In particular, their results apply to the classical Ces\`aro operator $C$ given by \eqref{eq.CesaroO} when it acts in the Korenblum growth spaces $A^{-\gamma}$ and $A^{-\gamma}_0$ for $\gamma > 0$.
Additional results which complement and extend their work can be found in \cite{ABR-R}.

For a detailed investigation of the optimal domain spaces $[C,H^p]$, for $1\leq p<\infty$, we refer to \cite{CR}. We now turn our attention to the setting when the $H^p$-spaces are replaced by the Korenblum spaces.
Note, whenever $g\in \cB$ is non-constant, that both of the optimal domain spaces $[V_{g_0}, A^{-\gamma}]$ and  $[V_{g_0}, A^{-\gamma}_0]$ are Banach spaces (cf. Proposition \ref{P.D1}). In view of Proposition \ref{P.Ugu}(ii), (iv), also $[C,A^{-\gamma}]=[T_{g_0}, A^{-\gamma}]$ and  $[C,A^{-\gamma}_0]=[T_{g_0}, A^{-\gamma}_0]$ are Banach spaces (see also \cite{ABR-R}). The following result is a consequence of Theorems 3.1 and 3.2 in \cite{ABR-R}.

\begin{prop}\label{P.DomainC} Let $\gamma>0$ and $g_0(z)=-{\rm Log}(1-z)$ for $z\in\D$. In each of the following cases the optimal domain space is genuinely larger than the original domain space.
	\begin{itemize}
		\item[\rm (i)] $A^{-\gamma}\subsetneqq [C,A^{-\gamma}]$ and $A^{-\gamma}\subsetneqq [V_{g_0},A^{-\gamma}]$.
		\item[\rm (ii)] $A^{-\gamma}_0\subsetneqq [C,A_0^{-\gamma}]$ and $A^{-\gamma}_0\subsetneqq [V_{g_0},A_0^{-\gamma}]$.
	\end{itemize}
	\end{prop}

\begin{proof} It is already known that $A^{-\gamma}\not=[C,A^{-\gamma}]$ and $A^{-\gamma}_0\not=[C,A_0^{-\gamma}]$; see  \cite[Theorems 3.1 and 3.2]{ABR-R}. So, by Proposition \ref{P.Ugu}(ii), (iv) we can conclude that also $A^{-\gamma}\not=[V_{g_0},A^{-\gamma}]$ and $A^{-\gamma}_0\not=[V_{g_0},A_0^{-\gamma}]$. For the stated inclusions we refer to Proposition \ref{P1}(iii).
	\end{proof}

The following result, \cite[Proposition 3.4]{ABR-R}, shows that the largest Korenblum space $A^{-\beta}$ that is contained in $[C,A^{-\gamma}]$ is $A^{-\gamma}$. The same is true for $A^{-\beta}_0$ and $[C,A^{-\gamma}_0]$. This property is particular for $C=T_{g_0}$ and is not valid for $T_g$ for general functions $g\in\cB$; see Proposition \ref{P.Ugu}(ii), (iv) and Example \ref{E2}(iii).

\begin{prop}\label{I} Let $\gamma>0$. For each $\beta>\gamma$, the space $A^{-\beta}\nsubseteq [C,A^{-\gamma}]$ and the space $A^{-\beta}_0\nsubseteq [C,A^{-\gamma}_0]$.
\end{prop}

The next result, \cite[Proposition 3.6]{ABR-R}, shows that the Banach spaces of analytic functions $A^{-\gamma}$ and $[C,A^{-\gamma}_0]$ on $\D$ are two \textit{non-comparable, proper} linear subspaces of the optimal domain space $[C,A^{-\gamma}]$. In particular, $C$ maps both of  these spaces into $A^{-\gamma}$.

\begin{prop}\label{J} Let $\gamma>0$. Both $A^{-\gamma}$ and $[C,A^{-\gamma}_0]$ are proper subspaces of $[C,A^{-\gamma}]$, with $[C,A^{-\gamma}_0]$ being closed. Moreover, $A^{-\gamma}$ and $[C,A^{-\gamma}_0]$ are non-comparable. That is,
	\[
	A^{-\gamma} \nsubseteq [C,A^{-\gamma}_0]\ \mbox{ and }\ [C,A^{-\gamma}_0]\nsubseteq A^{-\gamma}.
	\]
\end{prop}

Observe that the closedness of $[C,A^{-\gamma}_0]$ in $[C,A^{-\gamma}]$ is a special case of Proposition \ref{P1}(v) for $T=C$ and with $X=A^{-\gamma}$ and $Y=A^{-\gamma}_0$.

\begin{prop}\label{P.Notclosed} Let $\gamma>0$ and $g_0(z)=-{\rm Log}(1-z)$ for $z\in\D$.
	\begin{itemize}
		\item[\rm (i)] Both of the Cesàro operators $C\colon A^{-\gamma}\to A^{-\gamma}$ and $C\colon A^{-\gamma}_0\to A^{-\gamma}_0$ fail to have closed range.
		\item[\rm (ii)] Both of the operators $V_{g_0}\colon A^{-\gamma}\to A^{-\gamma}$ and $V_{g_0}\colon A^{-\gamma}_0\to A^{-\gamma}_0$ fail to have closed range.
	\end{itemize}
	\end{prop}

\begin{proof}
	(i) The Cesàro operator $C\in \cL(H(\D))$ is an isomorphism and satisfies
	\begin{equation}\label{Nn}
	C^{-1}(z^n)=(n+1)(1-z)z^n,\quad n\in\N_0.
	\end{equation}
Indeed, for fixed $n\in\N_0$, the function $f_n(z):=(n+1)(1-z)z^n$, for $z\in\D$, satisfies
\[
(Cf_n)(z)=\frac{1}{z}\int_0^z\frac{(n+1)(1-\xi)\xi^n}{1-\xi}\,d\xi=z^n,\quad n\in\N_0,
\]
from which \eqref{Nn} follows. It is clear from \eqref{Nn} that the range $C(A^{-\gamma}_0)\subseteq A^{-\gamma}_0$ contains the polynomials and hence, it is dense in $A_0^{-\gamma}$.
	Suppose that  $C\colon A^{-\gamma}_0\to A^{-\gamma}_0$ has closed range, that is,  $C(A^{-\gamma}_0)$ is a closed subspace of $A_0^{-\gamma}$. It follows that $C(A^{-\gamma}_0)= A^{-\gamma}_0$ and hence, $C\colon A^{-\gamma}_0\to A^{-\gamma}_0$ is a surjective isomorphism. This is a contradiction because $0$ belongs to the spectrum $\sigma(C;A^{-\gamma}_0)$; see \cite[Theorems 4.1, 5.1 and Corollaries 2.1, 5.1]{P},  \cite[Theorem 4.1]{AP}. So, $C\in \cL(A^{-\gamma}_0)$ fails to have closed range.
	
	Suppose now that $C\colon A^{-\gamma}\to A^{-\gamma}$ has closed range. By the open mapping theorem it follows that there exists $D>0$ such that $\|Cf\|_{-\gamma}\geq D\|f\|_{-\gamma}$ for all $f\in A^{-\gamma}$. Since $A^{-\gamma}_0\subseteq A^{-\gamma}$, this implies that $\|Cf\|_{-\gamma}\geq D\|f\|_{-\gamma}$ for all $f\in A^{-\gamma}_0$ and hence, that $C\colon A^{-\gamma}_0\to A^{-\gamma}_0$ has closed range. A contradiction.
	
	(ii) First observe,  by Proposition \ref{P.Ugu}(ii), (iv),   that $[C,A^{-\gamma}]=[V_{g_0},A^{-\gamma}]$ and $[C,A^{-\gamma}_0]=[V_{g_0},A^{-\gamma}_0]$ as linear spaces and topologically. Moreover, as noted above, both $[C,A^{-\gamma}]$ and $[C,A^{-\gamma}_0]$ are Banach spaces. By part (i) combined with Proposition \ref{P5} (or with Corollary \ref{C1}) it  follows necessarily that $A^{-\gamma}$ ($A^{-\gamma}_0$, resp.) is \textit{not} a closed subspace of $[C,A^{-\gamma}]$ (resp. of $[C,A^{-\gamma}_0]$).  Accordingly, $A^{-\gamma}$ ($A^{-\gamma}_0$, resp.) is not a closed subspace of $[V_{g_0},A^{-\gamma}]$ (resp. of $[V_{g_0},A^{-\gamma}_0]$) either. Since $[V_{g_0},A^{-\gamma}]$ and $[V_{g_0},A^{-\gamma}_0]$ are Banach spaces, again by Proposition \ref{P5} we can conclude that both of  the operators $V_{g_0}\colon A^{-\gamma}\to A^{-\gamma}$ and $V_{g_0}\colon A^{-\gamma}_0\to A^{-\gamma}_0$ fail to have closed range.
\end{proof}

The following fact is known. We include a proof  for the sake of completeness.

\begin{lemma}\label{L.K} Let $(X,\|\cdot\|_X)$ be an infinite dimensional  Banach space and $T\in \cL(X)$ be injective and compact. Then  $T$ does not have closed range in $X$.
	\end{lemma}

\begin{proof}
	Since $T\colon X\to X$ is injective, necessarily $T(X)$ is an infinite dimensional subspace of  $X$.
	Suppose that $T\in \cL(X)$ does have closed range. Then $T(X)$ is a closed subspace of $X$ and hence, by the open mapping theorem, the operator $T\colon X\to T(X)$ is an isomorphism. But, $T\in \cL(X)$ is compact. So, $T(B_X)\subseteq T(X)$ is a relatively compact $0$-neighbourhood of $T(X)$, which  implies that  $T(X)$ is finite dimensional. A contradiction.
\end{proof}

The function $g_0(z)=-{\rm Log}(1-z)$, for $z\in\D$, belongs to $\cB\setminus\cB_0$. The following result has similarities with Proposition \ref{P.DomainC}.

\begin{prop}\label{P.B0}Let $g\in \cB_0$ and   $\gamma>0$.
	\begin{itemize}
		\item[\rm (i)] Both of the operators $V_g\colon A^{-\gamma}\to A^{-\gamma}$ and $V_g\colon A^{-\gamma}_0\to A^{-\gamma}_0$ fail to  have closed range.
		\item[\rm (ii)] Each of the containments $A^{-\gamma}\subseteq [V_g, A^{-\gamma}]$ and  $A^{-\gamma}_0\subseteq [V_g, A^{-\gamma}_0]$ is proper. That is, the optimal domain space is genuinely larger than the original domain space.
			\item[\rm (iii)] Both of the operators $T_g\colon A^{-\gamma}\to A^{-\gamma}$ and $T_g\colon A^{-\gamma}_0\to A^{-\gamma}_0$ fail to  have closed range.
			\item[\rm (iv)] Each of the containments $A^{-\gamma}\subseteq [T_g, A^{-\gamma}]$ and  $A^{-\gamma}_0\subseteq [T_g, A^{-\gamma}_0]$ is proper. That is, the optimal domain space is genuinely larger than the original domain space.
		\end{itemize}
	\end{prop}

\begin{proof}
	(i) By Proposition \ref{P.Comp} both of the operators $V_g\colon A^{-\gamma}\to A^{-\gamma}$ and $V_g\colon A^{-\gamma}_0\to A^{-\gamma}_0$ are compact. On the other hand, the operators $V_g\colon A^{-\gamma}\to A^{-\gamma}$ and $V_g\colon A^{-\gamma}_0\to A^{-\gamma}_0$ are also injective. So, the result follows from Lemma \ref{L.K}.
	
	(ii) Since $[V_g, A^{-\gamma}]$ and  $[V_g, A^{-\gamma}_0]$ are Banach spaces, the result follows from part (i) combined with  Proposition \ref{P5} (or with Corollary \ref{C1}).
	
	(iii) Let $T$ be the operator given by \eqref{5B}. Since $T_g=T\circ V_g$ with $T\in \cL(\cE^\infty_{v}; A^{-\gamma})\cap \cL(\cE^0_{v}, A^{-\gamma}_0)$ (cf. Proposition \ref{P_C1} with $v(z):=(1-|z|)^\gamma$, for $z\in\D$, in which case $H^\infty_v=A^{-\gamma}$ and $H^0_v=A^{-\gamma}_0$) and both of the operators $V_g\colon A^{-\gamma}\to A^{-\gamma}$ and $V_g\colon A^{-\gamma}_0\to A^{-\gamma}_0$ are compact, also the operators $T_g\colon A^{-\gamma}\to A^{-\gamma}$ and $T_g\colon A^{-\gamma}_0\to A^{-\gamma}_0$ are compact. Moreover, both of the operators $T_g\colon A^{-\gamma}\to A^{-\gamma}$ and $T_g\colon A^{-\gamma}_0\to A^{-\gamma}_0$ are also injective. So, the result follows again from Lemma \ref{L.K}.
	
	(iv) Follows by arguing as in part (ii) and using part (iii).
\end{proof}

We conclude this section by giving some further classes of functions $g\in \cB$ for which the operators $V_g$ and $T_g$ fail to have closed range. In order to do this, we need some preliminary results.

\begin{lemma}\label{L.R1} Let $\gamma>0$ and $g\in\cB$ be non-constant.
	\begin{itemize}
		\item[\rm (i)] Suppose that   $V_g(A^{-\gamma})=\{f\in A^{-\gamma}:\ f(0)=0\}$. Then $A^{-\gamma}=[V_g, A^{-\gamma}]$.
		\item[\rm (ii)] Suppose that   $V_g(A^{-\gamma}_0)=\{f\in A^{-\gamma}_0:\ f(0)=0\}$. Then $A^{-\gamma}_0=[V_g, A^{-\gamma}_0]$.
		\end{itemize}
	\end{lemma}

\begin{proof}
	(i) We know that $A^{-\gamma}\subseteq [V_g, A^{-\gamma}]$.
	
	Let $f\in [V_g, A^{-\gamma}]$. Then $f\in H(\D)$ and $V_gf\in A^{-\gamma}$. Since $(V_gf)(0)=0$ (see \eqref{eq.V}), it follows that $V_gf\in \{u\in A^{-\gamma}:\ u(0)=0\}$ and so, by assumption, there exists $h\in A^{-\gamma}$ such that $V_gh=V_gf$. The injectivity of $V_g$ implies that $h=f$. Accordingly, $f\in A^{-\gamma}$.
	
	(ii) Follows by arguing as in part (i).
\end{proof}

\begin{lemma}\label{L.R2} Let $g\in\cB$ satisfy $|g'|>0$ on $\D$ and $\frac{1}{g'}\in H^\infty$.  Suppose, for $\gamma>0$,  that $V_g(A^{-\gamma}_0)=\{f\in A^{-\gamma}_0:\ f(0)=0\}$. Then $V_g(A^{-\gamma})=\{f\in A^{-\gamma}:\ f(0)=0\}$.
	\end{lemma}

\begin{proof} Let $f\in A^{-\gamma}$ satisfy $f(0)=0$. Select a sequence $(q_n)_{n\in\N}$ of polynomials satisfying $\|q_n\|_{-\gamma}\leq \|f\|_{-\gamma}$ for all $n\in\N$, and $q_n\to f$ in $H(\D)$ as $n\to\infty$, \cite{BS}. In particular,  $q_n(0)\to f(0)=0$ as $n\to\infty$, and hence, $p_n:=q_n-q_n(0)\to f$ in $H_0(\D)\subseteq H(\D)$ as $n\to\infty$. Therefore, $\|p_n\|_{-\gamma}\leq \|q_n\|_{-\gamma}+|q_n(0)|\leq M$ for all $n\in\N$ and some $M>0$. Recall that $V_g\colon H(\D)\to H_0(\D)$ is bijective and continuous. Moreover,  $V^{-1}_g\colon H_0(\D)\to H(\D)$ is given by  $h\mapsto \frac{h'}{g'}$ for $h\in H_0(\D)$. Since $H_0(\D)$ is a closed subspace of $H(\D)$, it is a Fr\'echet space and so, by the open mapping theorem, $V_g^{-1}$ is
	 continuous. Hence,  $V^{-1}_gp_n\to V^{-1}_gf$ in $H(\D)$ as $n\to\infty$. On the other hand, for each $n\in\N$ we have that $p_n(0)=0$ and $p_n\in A^{-\gamma}_0$. So, by the assumption and the injectivity of $V_g$ it follows that $V^{-1}_gp_n\in A^{-\gamma}_0$ for all $n\in\N$.
	
Now, observe that the operator $V_g\colon A_0^{-\gamma}\to \{f\in A^{-\gamma}_0:\ f(0)=0\}$ is continuous and bijective. By the open mapping theorem the inverse operator $V_g^{-1}\colon \{f\in A^{-\gamma}_0:\ f(0)=0\}\to A_0^{-\gamma}$ exists and  is continuous. So, there exists $D>0$ such that $\|V_g^{-1}p_n\|_{-\gamma}\leq D$ for all $n\in\N$. Accordingly, for each $n\in\N$, we have
\[
(1-|z|)^\gamma |(V_g^{-1}p_n)(z)|\leq D,\quad z\in\D.
\]
Since $V^{-1}_gp_n\to V^{-1}_gf$ in $H(\D)$ as $n\to\infty$, it follows that
\[
(1-|z|)^\gamma |(V_g^{-1}f)(z)|\leq D,\quad z\in\D.
\]
This implies that $V_g^{-1}f\in A^{-\gamma}$. So, $f=V_g(V_g^{-1}f)\in V_g(A^{-\gamma})$ and we can conclude that $\{f\in A^{-\gamma}:\ f(0)=0\}\subseteq V_g(A^{-\gamma})$.

The reverse inclusion $V_g(A^{-\gamma})\subseteq \{f\in A^{-\gamma}:\ f(0)=0\}$ is clear.
	\end{proof}

\begin{prop}\label{P.NoCl}  Let $\gamma>0$ and $g\in\cB$ satisfy $|g'|>0$ on $\D$ and $\frac{1}{g'}\in H^\infty$. Suppose that  there exists $w\in\C$ with $|w|=1$  such that $\lim_{r\to 1^-}|g'(rw )|(1-r)=0$. Then both of the operators $V_g\colon A^{-\gamma}\to A^{-\gamma}$ and $V_g\colon A^{-\gamma}_0\to A^{-\gamma}_0$ fail to have closed range.
	\end{prop}

\begin{proof}  Observe that the set $\cP_0$ of all polynomials on $\D$ vanishing at $0$ is contained in $V_g(A^{-\gamma}_0)$. Indeed, for a fixed $n\in\N$, the function $f(z):=\frac{nz^{n-1}}{g'(z)}$, for $z\in\D$, clearly belongs to $A^{-\gamma}_0$ because  $\frac{1}{g'}\in H^\infty$ and $nz^{n-1}\in A^{-\gamma}_0$. Moreover,
	\[
	(V_gf)(z)=\int_0^zf(\xi)g'(\xi)\,d\xi=\int_0^z n\xi^{n-1}\,d\xi=z^n,\quad z\in\D,
	\]
	and hence, $z^n\in V_g(A^{-\gamma}_0)$. Since $\{z^n:\ n\in\N\}\subseteq V_g(A^{-\gamma}_0)$, it follows that $\cP_0$ is contained in $V_g(A^{-\gamma}_0)$.
	
	We point out that  $ \cP_0$ is also dense in the closed subspace $\{f\in A_0^{-\gamma}:\ f(0)=0\}$ of $A^{-\gamma}_0$. To see this fix $h\in \{f\in A_0^{-\gamma}:\ f(0)=0\}$. Since $\cP$ is dense in $A^{-\gamma}_0$, there exists a sequence $(q_n)_{n\in\N}$ of polynomials such that $q_n\to h$ in $A^{-\gamma}_0$  as $n\to\infty$. In particular, also   $q_n(0)\to h(0)=0$ as $n\to\infty$. Therefore, the polynomials $p_n:=q_n-q_n(0)\in \cP_0$, for all $n\in\N$,  and satisfy  $p_n\to h$ in $A^{-\gamma}_0$  as $n\to\infty$.
	
	Since $\cP_0\subseteq V_g(A^{-\gamma}_0)\subseteq \{f\in A^{-\gamma}_0:\ f(0)=0\}$ and  $\cP_0$ is dense in $\{f\in A^{-\gamma}_0:\ f(0)=0\}$, it follows that also $V_g(A^{-\gamma}_0)$ is dense in $\{f\in A^{-\gamma}_0:\ f(0)=0\}$. Suppose that $V_g\colon A^{-\gamma}_0\to A^{-\gamma}_0$ has closed range, in which case $V_g(A^{-\gamma}_0)= \{f\in A^{-\gamma}_0:\ f(0)=0\}$. By Lemma \ref{L.R2} it follows that $V_g(A^{-\gamma})= \{f\in A^{-\gamma}:\ f(0)=0\}$.  In view of Lemma \ref{L.R1}(i) this implies that $A^{-\gamma}=[V_g, A^{-\gamma}] $ and hence, the operator  $V_g\colon A^{-\gamma}\to A^{-\gamma}$ has closed range (cf. Proposition \ref{P5} or Corollary \ref{C1}(ii)). But, $A^{-\gamma}=[V_g, A^{-\gamma}] $  is in  contradiction with Example \ref{E1}. So, we can conclude that both of the operators  $V_g\colon A^{-\gamma}\to A^{-\gamma}$ and $V_g\colon A^{-\gamma}_0\to A^{-\gamma}_0$ fail to  have closed range.
\end{proof}

We conclude with some relevant examples concerning optimal domain spaces of $V_g$ acting in Korenblum spaces.

\begin{example}\label{E2}\rm  (i) If $g$ is non-constant and  $g'\in H^\infty$ then, for each $\gamma >0$, we have $A^{-(\gamma+1)}\subseteq [V_g, A^{-\gamma}]$. Indeed, if $f\in A^{-(\gamma+1)}$, then also $fg'\in A^{-(\gamma+1)}$ because   $g'\in H^\infty$. Moreover, $g'\in H^\infty$ implies that $g\in \cB$. So, by Proposition \ref{P.Description}(i), it follows that $f\in  [V_g, A^{-\gamma}]$. In particular, $ [V_g, A^{-\gamma}]$ is genuinely larger than $A^{-\gamma}$.
	
	(ii) Let $g\in \cB$ satisfy $|g'|>0$ on $\D$ and $\frac{1}{g'}\in H^\infty$. Then, for each $\gamma >0$,  we have $[V_g, A^{-\gamma}]\subseteq A^{-(\gamma+1)}$. Indeed, if $f\in [V_g, A^{-\gamma}]$, then $f\in H(\D)$ and $fg'\in A^{-(\gamma+1)}$ (cf. Proposition \ref{P.Description}(i)). Since $\frac{1}{g'}\in H^\infty$, it follows that $f=\frac{1}{g'}(fg')\in A^{-(\gamma+1)}$.

	(iii) Let $g(z);=z^n$ for all $z\in\D$ and some $n\in\N$. Then, for each $\gamma >0$,  we have that  $A^{-(\gamma+1)}\subseteq [V_g, A^{-\gamma}]$ by part (i), as $g'\in  H^\infty$.
	
	If $g(z):=z$, for $z\in\D$, then  $g'(z)=1$ for all $z\in\D$ and so both $g', \frac{1}{g'}\in H^\infty$. By parts (i) and (ii) it follows that $[V_g, A^{-\gamma}]= A^{-(\gamma+1)}$.

	Let $g(z):=z^2$,  for  $z\in\D$, in which case  $g'(z)=2z$ for all $z\in\D$. Given  $f\in [V_g, A^{-\gamma}]$,  Proposition \ref{P.Description}(i) implies that $f\in H(\D)$ and $(fg')(z)=2zf(z)\in A^{-(\gamma+1)}$. Since $(zf(z))(0)=0$, it follows that  $zf(z)\in \cE_v^\infty$, where $v(r):=(1-r)^{\gamma+1}$
	for $r\in [0,1)$ and $\cE^\infty_v=\{h\in A^{-(\gamma+1)}:\ h(0)=0\}$. Hence, by Proposition \ref{P_C1}(i) we see that $f=T(zf(z))\in A^{-(\gamma+1)}$. So, $[V_g, A^{-\gamma}]\subseteq A^{-(\gamma+1)}$. Accordingly, $[V_g, A^{-\gamma}]= A^{-(\gamma+1)}$. Since $\frac{1}{g'}\not\in H^\infty$, this condition is sufficient (cf. part (ii)) but, not necessary.
	
	Finally, if   $g(z)=z^n$ for all $z\in\D$ (hence, $g'(z)=nz^{n-1}$ for all $z\in\D$) then, for any given $f\in [V_g, A^{-\gamma}]$, Proposition \ref{P.Description}(i) implies that $f\in H(\D)$ and $(fg')(z)=nf(z)z^{n-1}\in A^{-(\gamma+1)}$. To conclude that $f\in A^{-(\gamma+1)}$ it suffices to repeat $(n-1)$ times the argument above for the case $n=2$, that is, to apply  $(n-1)$ times  Proposition \ref{P_C1}(i). Accordingly, $[V_g, A^{-\gamma}]= A^{-(\gamma+1)}$ for all functions $g(z):=z^n$ with $n\in\N$.
	
	(iv) Let $g\in \cB$. For $0<\beta\leq \gamma$, we have seen  that $A^{-\beta}\subseteq A^{-\gamma}\subseteq [V_g, A^{-\gamma}]$.
	
	Now, let $0<\gamma<\beta$. Suppose $A^{-\beta}\subseteq [V_g, A^{-\gamma}]$. Then, for any $f\in A^{-\beta}$,  it follows (cf. Proposition \ref{P.Description}(i)) that $f\in H(\D)$ and $fg'\in A^{-(\gamma+1)}$. This implies that the multiplication operator $M_{g'}\colon H(\D)\to H(\D)$ satisfies $M_{g'}(A^{-\beta})\subseteq A^{-(\gamma+1)}$. Since $M_{g'}\in \cL(H(\D))$, it follows that the linear  operator $M_{g'}\colon A^{-\beta}\to A^{-(\gamma+1)}$ is  defined and has closed graph. Hence, $M_{g'}\in \cL(A^{-\beta},A^{-(\gamma+1)})$. By \cite[Proposition 3.1]{Cont-Diaz} this implies that
	\begin{equation}\label{c}
	\sup_{z\in\D}\frac{(1-|z|)^{\gamma+1}|g'(z)|}{(1-|z|)^\beta}<\infty.
	\end{equation}
	If $\beta>(\gamma+1)$, then \eqref{c} yields that  $\lim_{|z|\to 1^-}|g'(z)|=0$. Accordingly, $g'(z)=0$ for all $z\in\D$ and hence, $g$ is a constant function.
	If $\gamma<\beta\leq(\gamma+1)$, then \eqref{c} yields that $g'\in A^{-(\gamma+1-\beta)}$. In particular, for $\beta=\gamma+1$ we can conclude that $g'\in H^\infty$.
	
	(v) Let $g\in \cB$ and $\gamma<\beta<(\gamma+1)$. Then $A^{-\beta}\subseteq [V_g, A^{-\gamma}]$ if and only if  $g'\in A^{-(\gamma+1-\beta)}$.
	Indeed, if  $A^{-\beta}\subseteq [V_g, A^{-\gamma}]$, then  part (iv) implies that $g'\in A^{-(\gamma+1-\beta)}$. Conversely, suppose that $g'\in H^\infty$. Then, for each $f\in A^{-\beta}$, we have that
	\begin{align*}
	&\sup_{z\in \D}(1-|z|)^{\gamma+1}|(fg')(z)|= \\
	&\quad\sup_{z\in\D}(1-|z|)^\beta|f(z)|(1-|z|)^{\gamma+1-\beta}|g'(z)|\leq \|f\|_{-\beta}\|g'\|_{-(\gamma+1-\beta)}<\infty.
	\end{align*}
	Accordingly, $fg'\in A^{-(\gamma+1)}$. By Proposition \ref{P.Description}(i) we can conclude that $f\in [V_g, A^{-\gamma}]$.
	
	(vi) Observe,  for $0<\beta<(\gamma+1)$, that the condition $g'\in A^{-(\gamma+1-\beta)}$ implies that $g\in \cB_0$.
	\end{example}

\vspace{.1cm}

\textbf{Acknowledgements.} The research of A.A. Albanese was partially supported by GNAMPA of INDAM.
The research of J. Bonet was partially supported by the project
PID2020-119457GB-100 funded by MCIN/AEI/10.13039/501100011033 and ``ERFD A way of making Europe'', and by Generalitat Valenciana project CAICO/2023/242. 


\bibliographystyle{plain}

\end{document}